\documentclass[a4paper,12pt]{article}
\PassOptionsToPackage{dvipsnames}{xcolor}
\usepackage{todonotes}
\usepackage{graphicx}
\usepackage{setspace}

\usepackage{latexsym} 

\usepackage[utf8]{inputenc}
\usepackage[english]{babel}

\usepackage{amsfonts,amsmath,amssymb,amsbsy,amsthm,enumerate}

\usepackage[mathscr]{euscript}
\usepackage{pstricks}
\usepackage{verbatim}
\usepackage{color}

\newtheorem{theorem}{Theorem}[section]

\newtheorem{conjecture}{Conjecture}
\newtheorem{proposition}[theorem]{Proposition}
\newtheorem{lemma}[theorem]{Lemma}
\newtheorem{corollary}[theorem]{Corollary}

\newtheorem{fact}{Fact}
\newtheorem{property}{Property}
\newtheorem*{gallai-conj}{Gallai's Conjecture (1966)}
\newtheorem*{hajos-conj}{Hajós's Conjecture (1968)}

\newcommand{\pn}{{\rm pn}}
\newcommand{\cn}{{\rm cn}}
\newcommand{\D}{\mathcal{D}}
\usepackage{mathtools} 

\newcommand{\lift}[1]{\widehat{#1}}

\newcommand{\floor}[1]{\lfloor #1 \rfloor}

\usepackage{tikz}
\usetikzlibrary{calc,positioning,decorations.pathmorphing,decorations.pathreplacing}
\tikzset{emp/.style={double distance = 0.3ex}}
\tikzset{oriented/.style={->,shorten >= 1.5pt}}
\usepackage{./Figures/figstyle}
\usepackage{subcaption}
\def\D{\mathcal{D}}

\usepackage{graphicx}

\makeatletter
\DeclareFontFamily{U}{tipa}{}
\DeclareFontShape{U}{tipa}{m}{n}{<->tipa10}{}
\newcommand{\arc@char}{{\usefont{U}{tipa}{m}{n}\symbol{62}}}%

\newcommand{\arc}[1]{\mathpalette\arc@arc{#1}}

\newcommand{\arc@arc}[2]{%
  \sbox0{$\m@th#1#2$}%
  \vbox{
    \hbox{\resizebox{\wd0}{\height}{\arc@char}}
    \nointerlineskip
    \box0
  }%
}
\makeatother

\newcounter{claimcount}
\newenvironment{claim}{\medskip\refstepcounter{claimcount}\noindent\textbf{Claim \arabic{claimcount}.}}{}

\usepackage{etoolbox}
\AtBeginEnvironment{lemma}{\setcounter{claimcount}{0}}
\AtBeginEnvironment{theorem}{\setcounter{claimcount}{0}}

\usepackage{pgfplots}

\usepackage{floatrow}
\usepackage[font=small,labelfont=bf]{caption}
\newcommand{\Note}[2][]
{}

\title{On Gallai's and Haj\'os' Conjectures for graphs with treewidth at most $3$}
\author{F. Botler\textsuperscript{1} \hspace{.5cm} M. Sambinelli\textsuperscript{2} \hspace{.5cm}
	R. S. Coelho\textsuperscript{3} \hspace{.5cm} O. Lee\textsuperscript{2}\\
	{\footnotesize \textsuperscript{1}Facultad de Ciencias F\'isicas y Matematicas}\vspace{-.2cm}\\
	{\footnotesize Universidad de Chile}\\
	{\footnotesize \textsuperscript{2}Instituto de Computação}\vspace{-.2cm}\\
	{\footnotesize Universidade Estadual de Campinas}\\
	{\footnotesize \textsuperscript{3}Instituto Federal do Norte de Minas Gerais}}

\begin{document}

\maketitle

\begin{abstract}
	A path (resp. cycle) decomposition of a graph \(G\) is a set of edge-disjoint paths (resp. cycles) of \(G\)
	that covers the edge set of \(G\).
	Gallai (1966) conjectured that every graph on \(n\) vertices admits a path decomposition
	of size at most \(\lfloor (n+1)/2\rfloor\),
	and Haj\'os (1968) conjectured that every Eulerian graph on \(n\) vertices admits a cycle decomposition
	of size at most \(\lfloor (n-1)/2\rfloor\).
	Gallai's Conjecture was verified for many classes of graphs.
	In particular,
	Lovász (1968) verified this conjecture for graphs with at most one vertex of even degree,
	and Pyber (1996) verified it for graphs in which every cycle contains a vertex of odd degree.
	Haj\'os' Conjecture, on the other hand, was verified only for graphs with maximum degree~\(4\)
	and for planar graphs.
	In this paper, we verify Gallai's and Haj\'os' Conjectures for graphs with treewidth at most $3$.
	Moreover, we show that the only graphs with treewidth at most~\(3\)
	that do not admit a path decomposition of size at most \(\lfloor n/2\rfloor\)
	are isomorphic to \(K_3\) or \(K_5-e\).
	Finally, we use the technique developed in this paper to present
	new proofs for Gallai's and Haj\'os' Conjectures for graphs with maximum degree at most~\(4\),
	and for planar graphs with girth at least~\(6\).

\end{abstract}

\let\thefootnote\relax\footnote{
This research has been partially supported by CNPq
Projects (Proc. 477203/2012-4 and {456792/2014-7}), Fapesp Project
(Proc. 2013/03447-6).
F. Botler is partially supported by CAPES (Proc. 1617829),
Millenium Nucleus Information and Coordination in Networks (ICM/FIC RC 130003),
and CONICYT/FONDECYT/POSTDOCTORADO 3170878.
M. Sambinelli is supported by CNPq (Proc. 141216/2016-6).
O. Lee is supported by CNPq (Proc. 311373/2015-1 and 477692/2012-5).
  e-mails:
  fbotler@ime.usp.br (F. Botler),
  msambinelli@ic.unicamp.br (M. Sambinelli),
  rcoelho@ime.usp.br (R. S. Coelho),
  lee@ic.unicamp.br (O. Lee)}

\section{Introduction}\label{sec:introduction}

In this paper, all graphs considered are simple, i.e., contain no loops or multiple edges.
A \emph{decomposition} \(\D\) of a graph \(G\) is a set \(\{H_1,\ldots,H_k\}\)
of edge-disjoint subgraphs of \(G\) that cover the edge set of \(G\).
We say that \(\D\) is a \emph{path} (resp. \emph{cycle}) \emph{decomposition} if~\(H_i\) is a path (resp. cycle) for \(i=1,\ldots,k\).
We say that a path (resp. cycle) decomposition~\(\D\) of a graph (resp. an Eulerian graph)~\(G\) is \emph{minimum}
if for any path (resp. cycle) decomposition~\(\D'\) of~\(G\) we have \(|\D| \leq |\D'|\).
The size of a minimum path (resp. cycle) decomposition
is called the \emph{path number} (resp. \emph{cycle number}) of \(G\),
and is denoted by \(\pn(G)\) (resp. \(\rm{cn}(G)\)).
In this paper, we focus in the following conjectures concerning minimum path and cycle numbers of graphs
(see \cite{Bondy14,Lovasz68}).

\begin{conjecture}[Gallai, 1966]\label{conj:gallai}
	If \(G\) is a connected graph,
	then \(\pn(G)\leq \big\lfloor \frac{|V(G)|+1}{2} \big\rfloor\).
\end{conjecture}

\begin{conjecture}[Haj\'os, 1968]\label{conj:hajos}
	If \(G\) is an Eulerian graph,
	then \(\cn(G)\leq \big\lfloor \frac{|V(G)|-1}{2} \big\rfloor\).
\end{conjecture}

In 1968, Lov\'asz~\cite{Lovasz68} proved that a graph with \(n\) vertices can be decomposed
into at most \(\lfloor n/2\rfloor\) paths and cycles.
The next theorem is a directed consequence of this result.

\begin{theorem}[Lov\'asz, 1968]\label{thm:lovasz}
	If \(G\) is a graph with \(n\) vertices and contains at most one vertex of even degree,
	then \(\pn(G) = \lfloor n/2\rfloor\).
\end{theorem}

Pyber~\cite{Pyber96} extended Theorem~\ref{thm:lovasz} as follows.

\begin{theorem}[Pyber, 1996]\label{thm:pyber}
	If \(G\) is a graph with \(n\) vertices in which every cycle contains a vertex of odd degree,
	then \(\pn(G)\leq\lfloor n/2\rfloor\).
\end{theorem}

In 2005, Fan~\cite{Fan05} extended Theorem~\ref{thm:lovasz} even more, but Conjecture~\ref{conj:gallai} is still open.
Recently, one of the authors~\cite{BotlerJimenez2015+} verified Conjecture~\ref{conj:gallai}
for a family of even regular graphs with high girth condition,
and Jim\'enez and Wakabayashi~\cite{JimenezWakabayashi14+} verified it
for a family of triangle-free graphs.
For more results concerning Conjecture~\ref{conj:gallai},
we refer the reader to~\cite{DeKo00,Donald80,FavaronKouider88,GengFangLi15}.
Although these conjectures seems very similar,
Conjecture~\ref{conj:hajos} was only verified
for graphs with maximum degree~\(4\)~\cite{GranvilleMoisiadis87}
and for planar graphs~\cite{Seyffarth92}.

The technique presented in this paper showed to be useful
to deal with both Conjectures~\ref{conj:gallai} and~\ref{conj:hajos}.
Let \(G\) be a counter-example for Conjecture~\ref{conj:gallai}
with a minimum number of vertices.
Our technique consists in finding a subgraph \(H\) of \(G\)
such that, for some positive integer \(r\),
the graph \(G'=G-E(H)\) contains at most \(|V(G)|-2r\) non-isolated vertices, and \(\pn(H)\leq r\).
Moreover, we show how to obtain \(H\) in such a way that for every component~\(C'\) of~\(G'\)
we have \(\pn(C') \leq \lfloor |V(C')|/2\rfloor\).
Therefore, the decomposition~\(\D\) of \(G\) obtained by joining minimum path
decompositions of \(H\) and \(G'\)
is such that \(|\D|\leq\lfloor |V(G)|/2\rfloor\).
The graph~\(H\) is called an \emph{\(r\)-reducing subgraph} and is discussed in Section~\ref{sec:reducing}.
As a byproduct of our main result (Theorem~\ref{theorem:treewidth-most-three}),
the only graphs with treewidth at most \(3\) and path number exactly \((n+1)/2\) are isomorphic to
\(K_3\) and \(K_5^-\) (the graph obtained from \(K_5\) by removing exactly one edge).
For Conjecture~\ref{conj:hajos} the procedure is analogous.

The main contributions of this paper are the following.
We verify Conjecture~\ref{conj:gallai} for graphs with treewidth at most \(3\),
graphs with maximum degree at most \(4\),
and planar graphs with girth at least~\(6\).
In fact, for all these cases, we prove the Conjecture~\ref{conj:strong-gallai}, which is a strengthening of Conjecture~\ref{conj:gallai}.
Also, we verify Conjecture~\ref{conj:hajos} for graphs with treewidth at most \(3\)
and present a new proof for the case of graphs with maximum degree at most~\(4\).
\begin{conjecture}\label{conj:strong-gallai}
	Let \(G\) be a connected graph with \(n\) vertices.
	If \(|E(G)| \leq (n-1)\big\lfloor{n}/{2} \big\rfloor\), then \(\pn(G) \leq \big\lfloor {n}/{2} \big\rfloor\).
        Otherwise, \(pn(G) = \big\lceil {n}/{2} \big\rceil\).
\end{conjecture}

Extended abstracts of parts of this work~\cite{BoSa-LAGOS,BoSaCoLe-ETC} were accepted to LAGOS~2017 and 
to the Brazilian Computer Society Conference (CSBC 2017).
While writing this paper, we learned that
Bonamy and Perrett~\cite{BonamyPerrett16+} verified Conjecture~\ref{conj:gallai} for graphs with maximum degree \(5\).
However, since the advance of the state-of-the-art of Conjecture~\ref{conj:gallai} in this direction is very recent,
we believe that both techniques and proofs are important to the literature.

This work is organized as follows.
In Section~\ref{sec:reducing}, we define reducing subgraphs,
present some technical lemmas,
and confirm Conjecture~\ref{conj:gallai} for planar graphs with girth at least \(6\).
In Section~\ref{sec:treewidth}, we settle Conjectures~\ref{conj:gallai} and~\ref{conj:hajos} for graphs with treewidth at most~\(3\)
and, in Section~\ref{sec:max4}, we present new proofs for Conjectures~\ref{conj:gallai} and~\ref{conj:hajos}
for graphs with maximum degree at most \(4\).
Finally, in Section~\ref{sec:conclusion}, we give some concluding remarks.

\subsection*{Notation}
The basic terminology and notation used in this paper are standard (see, e.g.~\cite{Di10}).
All graphs considered here are finite and have no loops nor multiple edges.
Let \(G=(V,E)\) be a graph.
A path \(P\) in \(G\) is a sequence of distinct vertices \(P = v_0v_1\cdots v_\ell\)
such that~\(v_iv_{i+1}\in E\).
It is also convenient to refer to a path \(P=v_0v_1\cdots v_\ell\)
as the subgraph of~\(G\) induced by the edges~\(v_iv_{i+1}\), for~\(i=1,\ldots,\ell-1\).
For ease of notation, given an edge \(xy\),
we denote by \(G+xy\) the graph \((V\cup\{x,y\},E\cup\{xy\})\), and by \(G-xy\) the graph \((V,E\setminus\{xy\})\).
If \(E'\) is a set of edges,
then \(G+E'\) (resp. \(G-E'\)) denotes \(G+\sum_{e\in E'} e\) (resp. \(G-\sum_{e\in E'}e\)).
Given two (not necessarily disjoint) graphs \(G\) and \(H\),
we use \(G+H\) to denote the graph \(\big(V(G)\cup V(H),E(G)\cup E(H)\big)\).
We write \(G_1 \simeq G_2\) (resp. \(G_1 \not\simeq G_2\)) to denote
that \(G_1\) is isomorphic (resp. non-isomorphic) to \(G_2\).
Given a set \(U\) and an element \(e\), we define \(U + e = U \cup \{e\}\) and \(U - e = U \backslash \{e\}\).

In this paper, we frequently count the isolated vertices \(V_i\) in a graph \(G\) after the removal of the edges of a subgraph \(H\subset G\).
To avoid the introduction of more notation,
when clear from context,
the graph \(G-E(H)\) denotes
either the graph \(\big(V(G),E(G)-E(H)\big)\) or the graph \(\big(V(G)-V_i,E(G)-E(H)\big)\).

Let \(G\) be a connected graph.
We say that a set \(S\subseteq E(G)\) of edges of \(G\)
is an \emph{edge separator} if \(G-S\) is disconnected.
If \(S\) is a minimal edge separator, i.e.,
\(S\) is an edge separator, but \(S'\) is not an edge separator
for every \(S'\subset S\) with \(S'\neq S\),
we say that \(S\) is an \emph{edge-cut}.

The figures in this paper are depicted as follows.
Solid edges and full vertices illustrate edges and vertices that are present in the graph,
while dashed edges and empty vertices illustrate edges and vertices that may be in the graph,
and loosely dotted edges illustrate edges that are not present in the graph.
Straight edges illustrate simple edges, while snake edges illustrate paths with possible internal vertices.

\begin{figure}[h]
\floatbox[{\capbeside\thisfloatsetup{capbesideposition={right,center},capbesidewidth=10cm}}]{figure}[\FBwidth]
{\caption{	the vertices $a$, $c$, $d$ illustrate full/present vertices;
		the vertex $b$ illustrates an empty/possible vertex;
		the edge $ad$ illustrates a solid/present edge;
		the edges $ab$, $bc$, and $bd$ illustrate dashed/possible edges;
		the edge $ac$ illustrates a loosely dotted/non-present edge;
		the edge $cd$ illustrates a snake edge/path.
	}\label{fig:test}}
{\scalebox{.8}{\input{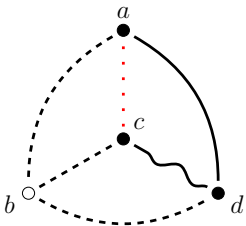}}}
\end{figure}

\section{Reducing subgraphs}\label{sec:reducing}

In this section, we define reducing subgraphs and present some results that allow us to deal with them
(see Lemma~\ref{lemma:reducing+triangles+K5-}).

Let \(G\) be a graph and let \(H\)  be a subgraph of \(G\).
Given a positive integer~\(r\),
we say that~\(H\) is an \emph{\(r\)-reducing subgraph of} \(G\)
if \(G-E(H)\) has at least \(2r\) isolated vertices
and \(\pn(H)\leq r\).
If \(H\) is a \(1\)-reducing subgraph of \(G\), then we say that~\(H\)
is a \emph{reducing path} of~\(G\).
We say that \(H\) is a \emph{reducing subgraph of} \(G\)
if \(H\) is an \(r\)-reducing subgraph of \(G\) for some positive integer \(r\).
We say that a graph \(G\) with \(n\) non-isolated vertices is a \emph{Gallai graph}
if \(\pn(G) \leq \lfloor n/2\rfloor\).
Note that, in this case,~\(G\) is also an \(\lfloor n/2\rfloor\)-reducing subgraph of itself.
The next lemma formalizes the relation between Gallai graphs and \(r\)-reducing subgraphs.
We also observe that since \(\pn(K_3) = 2\) and \(\pn(K_5) = \pn(K_5^-) = 3\),
the graphs \(K_3\), \(K_5\) and \(K_5^-\) are not Gallai graphs.

\begin{lemma}\label{lemma:reducing}
	Let \(G\) be a graph, and let \(H\subseteq G\) be a reducing subgraph of \(G\).
	If \(G-E(H)\) is a Gallai graph, then \(G\) is a Gallai graph.
\end{lemma}

\begin{proof}
	Let \(G\) and \(H\) be as in the statement,
	where \(H\) is an \(r\)-reducing subgraph of \(G\).
	By the definition of \(r\)-reducing subgraph,
	\(G-E(H)\) has at least \(2r\) isolated vertices,
	and since \(G-E(H)\) is a Gallai graph,
	\(\pn(G-E(H)) \leq \lfloor (n-2r)/2\rfloor = \lfloor n/2\rfloor - r\).
	Thus, there is a path decomposition \(\D'\) of \(G-E(H)\)
	with size at most \(\lfloor n/2\rfloor - r\).
	Since~\(H\) is an \(r\)-reducing subgraph,
	there is a path decomposition \(\D_H\) of \(H\) with size at most \(r\), and hence
	\(\D'\cup\D_H\) is a path decomposition of \(G\)
	with size at most \(\lfloor n/2\rfloor\).
\end{proof}

Now we are able to verify Conjecture~\ref{conj:gallai} for planar graphs with girth at least~\(6\).
For that, we use the fact that every connected planar graph with \(n\) vertices and girth at least~\(6\),
contains at least three vertices of degree at most~\(2\).
Indeed, let \(G\) be a connected planar graph with girth at least \(6\),
and let \(n'\) be the number of vertices of \(G\) with degree at most~\(2\).
Since every vertex of \(G\) has degree at least \(1\),
if \(n' \leq 2\), then \(2|E(G)| = \sum_{v\in V(G)} d(v) \geq n' + 3(n-n') = 3n -2n' \geq 3n-4\).
On the other hand, since \(G\) has girth at least~\(6\),
the boundary walk of each face of \(G\) contains at least \(6\) edges.
Thus, \(2|E(G)| \geq 6f\), where \(f\) is the number of faces of \(G\).
By Euler's formula, we have \(n+f-|E(G)| = 2\),
which implies \(6 = 3n+3f-3|E(G)|  \leq 2|E(G)| + 4 + |E(G)| - 3|E(G)| = 4\),
a contradiction.

\begin{theorem}\label{thm:gallai-planar-girth6}
	Every connected planar graph with girth at least \(6\)
	is a Gallai graph.
\end{theorem}

\begin{proof}
	Suppose, for a contradiction, that the statement does not hold,
	and let \(G\) be a counter-example for the statement with a minimum number of vertices.
	As observed above, \(G\) contains at least three vertices of degree at most \(2\).
	Let \(P'\) be a shortest path in \(G\) joining two of these vertices, say \(u\) and \(v\), and
	let \(S\) be the set of edges of \(G-E(P')\) incident to \(u\) or \(v\).
        Put~\(P = P'+S\).
	Since \(u\) and \(v\) have degree at most \(2\),
	\(S\) contains at most one edge incident to each~\(u\) and~\(v\),
        and since \(G\) has girth at least \(6\),
        \(P\) is a path.
	Moreover, \(u\) and \(v\) are isolated in \(G' = G - E(P)\)
	and hence~\(P\) is a reducing path.
        Note that each component of \(G'\) is a planar graph with girth at least~\(6\).
	By the minimality of \(G\), each component of \(G'\) is a Gallai graph,
        hence \(G'\) is a Gallai graph.
        Therefore, by Lemma~\ref{lemma:reducing}, \(G\) is a Gallai graph,
	a contradiction.
\end{proof}

Now let \(G\) be a Gallai graph on \(n\) non-isolated vertices
and let \(\D\) be a minimum path decomposition of \(G\).
Since each path in \(G\) contains at most \(n-1\),
we have \(|E(G)| = \sum_{P\in\D}|E(P)| \leq \sum_{P\in\D}(n-1) \leq \lfloor n/2\rfloor(n-1)\).
Note that this inequality holds whenever~\(n\) is even.
In the case where \(n\) is odd, we say that a graph \(G\) with \(n\)
vertices is \emph{quasi-complete} (or an \emph{odd semi-clique} -- see~\cite{BonamyPerrett16+})
if \(|E(G)| > \lfloor n/2\rfloor (n-1)\).
Therefore, it is clear that no quasi-complete graph is a Gallai graph.
The quasi-complete graphs are precisely the graphs obtained from~\(K_n\) (with \(n\) odd)
by removing at most \(\lfloor n/2\rfloor -1\) edges.
A direct implication of our main result is
that the only non-Gallai partial \(3\)-tree
are the quasi-complete partial \(3\)-trees, i.e., the complete graph \(K_3\) and \(K_5^-\).
Our main technique consist in finding reducing subgraphs of given graphs.
The following results allow us to construct \(r'\)-reducing subgraphs from \(r\)-reducing subgraphs
and graphs isomorphic to \(K_3\), \(K_5\), and \(K_5^-\).
Note that the following results require only the graph \(G\) to be connected
and need no other property as, for example, being a partial \(3\)-tree.

The proof of the next lemma follows the proof of Lemma 3.2 in~\cite{BotlerJimenez2015+}.
Given a cycle \(C\) in a graph \(G\),
a \emph{chord} of \(C\) is an edge in \(G-E(C)\) that joins two distinct vertices of \(C\).
Given a path \(P = x_0x_1\cdots x_k\), if \(i<j\), we denote by \(P(x_i,x_j)\)
the subpath \(x_ix_{i+1}\cdots x_j\).

\begin{lemma}\label{lemma:cycle+path}
	Let \(G\) be a connected graph that admits a decomposition into a path \(P\) and a cycle \(C\).
	\begin{enumerate}[(i)]
	\item if \(P\) contains at most one chord of \(C\),
			then \(G\) admits a decomposition into two paths \(P_1\) and \(P_2\)
			such that \(P_1\) contains exactly one edge of \(C\); and
	\item if \(C\) has length at most \(5\) and
			\(P\) contains at most three chords of \(C\),
			then \(\pn(G) = 2\).
	\end{enumerate}
\end{lemma}
\begin{proof}
	Let \(G\), \(P\), and \(C\) be as in the statement.
	Let \(P = x_0x_1 \cdots x_k\) and \(C = y_0 \cdots y_ly_0\).
	Let \(z_0, \ldots, z_s\) be the vertices in \(V(C) \cap V(P)\)
	in the order that they appear in \(x_0 \cdots x_k\), and
	suppose without loss of generality that \(z_0 = y_0\).
	
	\begin{claim}\label{claim:cycle-path}
		If \(\{y_1,y_l\}\not\subset V(P)\) or 
		\(P(z_{i-1},z_i)\) has length at least~\(2\) for some \(z_i \in \{y_1, y_l\}\),
		then \(G\) can be decomposed into two paths \(P_1\) and \(P_2\) such that
		\(P_1\) contains exactly one edge of \(C\).
	\end{claim}
	\begin{proof}
        If there exists a vertex \(y \in \{y_1, y_{l}\}\) such that \(y \notin V(P)\),
        then \(P_1 = P - E(P(x_0, z_0)) + z_0y\) and \(P_2 = G - E(P_1)\) decompose \(G\) as desired (see Figure~\ref{fig:cycle-path-claim1}).
	If~\(P(z_{i-1}, z_i)\) has length at least \(2\) for some \(z_i \in \{y_1, y_l\}\), 
	then there is a neighbor \(z'\) of~\(z_i\) in~\(P(z_{i-1}, z_i)\).
        Note that \(z'\) is not a vertex of \(C\).
	Thus, \(P_1 = P -E(P(x_0,z_0)) - z'z_i + z_0y\) and \(P_2 = G - E(P_1)\) decompose \(G\) as desired (see Figure~\ref{fig:cycle-path-claim2}).
        \end{proof}

	\begin{figure}[h]
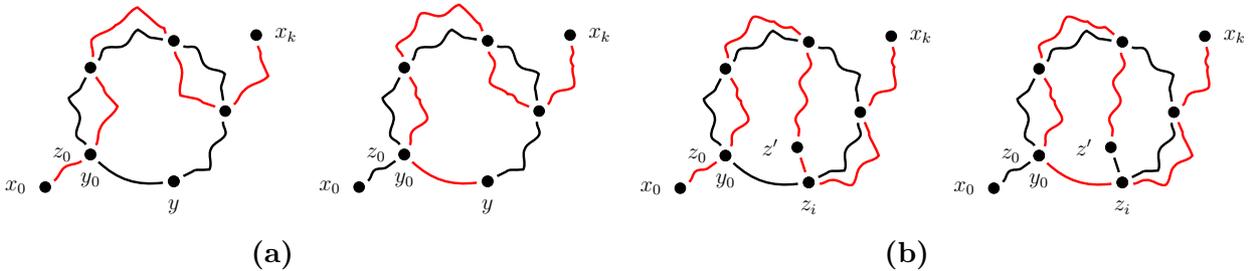

	\centering
  \hspace{ -1.5cm}
		\begin{subfigure}[b]{.45\linewidth}
		\centering
		\scalebox{.7}{\input{Figures/cycle-path1-before.tikz}}%
		\scalebox{.7}{\input{Figures/cycle-path1-after.tikz}}

		\caption{}\label{fig:cycle-path-claim1}
		\end{subfigure}%
    \hspace{ 1cm}
		\begin{subfigure}[b]{.45\linewidth}
		\centering\scalebox{.7}{\input{Figures/cycle-path2-before.tikz}}%
		\centering\scalebox{.7}{\input{Figures/cycle-path2-after.tikz}}
		\caption{}\label{fig:cycle-path-claim2}
		\end{subfigure}%

                \caption{The illustration on the left of figures (a) and (b) shows the path $P$ and the cycle $C$ in red and black, respectively, while
                  the illustration on the right shows the paths $P_1$ and $P_2$ in red and black, respectively.}
\label{fig:case9}
\end{figure}

        Now we prove item (i).
        Suppose that \(P\) contains at most one chord of \(C\).
        If \(\{y_1,y_l\}\not\subset V(P)\), then the result follows by Claim~\ref{claim:cycle-path}.
	Thus, let \(\{z_i,z_j\} = \{y_1,y_l\}\).
        Since \(P\) contains at most one chord of \(C\),
        at least one between \(P(z_{i-1}, z_i)\) or \(P(z_{j-1}, z_j)\) has length at least~\(2\).
        Again, by Claim~\ref{claim:cycle-path}, the result follows.
	This concludes the proof of Item (i).

	Now we prove item (ii).
	Suppose that \(C\) has length at most \(5\) and that \(P\) contains at most \(3\) chords of \(C\).
	If \(C\) has length \(3\), then \(P\) contains no chord of \(C\) and the result follows by item (i).
	Thus, suppose that \(C\) has length \(4\).
	By Claim~\ref{claim:cycle-path}, we may assume that \(\{y_1,y_3\}\subset V(P)\), and hence 
        let \(z_i \in \{y_1,y_3\}\).
	If \(P(z_{i-1},z_i)\) has length \(1\),
	then \(P(z_{i-1},z_i) = y_1y_3\) and \(z_{i-1}\in\{y_1,y_3\}-z_i\).
	Thus, if \(z_i=y_1\) and \(z_j=y_3\), then
	either \(P(z_{i-1},z_i)\) or \(P(z_{j-1},z_j)\) has length at least \(2\),
	and the result follows by item (i).

\begin{figure}[h]
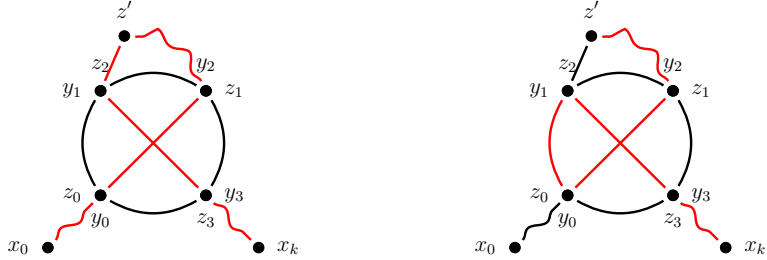

	\centering
	\begin{subfigure}[b]{.25\linewidth}
	\centering\scalebox{.7}{\input{Figures/cycle-path3-before.tikz}}

	\end{subfigure}%
  \hspace{ 2cm}
	\begin{subfigure}[b]{.25\linewidth}
	\centering\scalebox{.7}{\input{Figures/cycle-path3-after.tikz}}

	\end{subfigure}%

                \caption{The illustration on the left shows the path $P$ and the cycle $C$ of length \(4\) in red and black, respectively, while
                  the illustration on the right shows the paths $P_1$ and $P_2$ in red and black, respectively.}
	\label{fig:case11}
\end{figure}

	Therefore, we may assume that \(C\) has length \(5\).
	By Claim~\ref{claim:cycle-path}, we can assume that \(\{y_1,y_4\}\subset V(P)\).
	Let \(\{z_i,z_j\}=\{y_1,y_4\}\),
	Suppose, without loss of generality, that \(i < j\)
	and that \(z_i = y_1\) and \(z_j=y_4\).
	By Claim~\ref{claim:cycle-path}, \(P(z_{i-1}, z_i)\) and \(P(z_{j - 1}, z_j)\) have length \(1\),
        i.e., \(P(z_{i-1}, z_i) = z_{i-1}z_i\) and \(P(z_{j-1}, z_j)=z_{j-1}z_j\) are two different chords of \(C\) in \(P\).
	We divide this proof into two cases, depending on whether \(z_iz_j\notin E(P)\) or \(z_iz_j\in E(P)\).
	First, suppose that \(z_iz_j \notin E(P)\).
	Since \(P(z_{i-1},z_i)\) and \(P(z_{j-1},z_j)\) have length \(1\),
	\(z_{i-1} = y_3\) and \(y_{j-1} = y_2\).
	Since there are no other vertex in \(C\),
	we have \(i=2\) and \(j=4\).
	Hence, \(z_1 = y_3\), \(z_2 = y_1\), \(z_3 = y_2\), and \(z_4 = y_4\).
	Since \(y_1y_2 \in E(C)\), the subpath \(P(z_2,z_3)\) has length at least \(2\).
	Let~\(z'\) be the neighbor of \(y_1\) in \(P(z_2,z_3)\),
        and put \(P_1 = z'y_1y_0 + P(z_0,z_1) + y_3y_2y_4 + P(z_4,x_k)\),
	and \(P_2 = G - E(P_1)\).
        Again, \(\{P_1, P_2\}\) is a path decomposition of \(G\) as desired (see Figure~\ref{fig:case12}).

\begin{figure}[h]
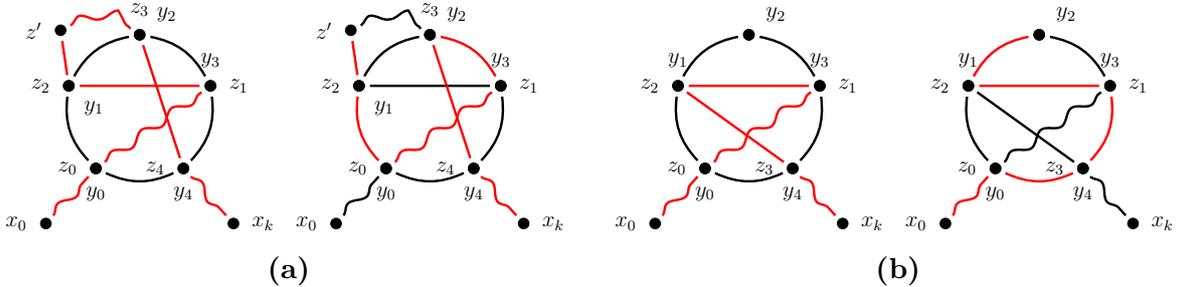

	\centering
	\begin{subfigure}[b]{.5\linewidth}
	\centering\scalebox{.7}{\input{Figures/cycle-path4-before.tikz}}%
	\centering\scalebox{.7}{\input{Figures/cycle-path4-after.tikz}}
	\caption{}\label{fig:case12}
	\end{subfigure}%
	\begin{subfigure}[b]{.5\linewidth}
	\centering\scalebox{.7}{\input{Figures/cycle-path5-before.tikz}}%
	\centering\scalebox{.7}{\input{Figures/cycle-path5-after.tikz}}
	\caption{}\label{fig:case13}
	\end{subfigure}%

        \caption{The illustration on the left of figures (a) and (b) shows the path $P$ and the cycle $C$ of length \(5\) in red and black, respectively, while
                 the illustration on the right shows the paths $P_1$ and $P_2$ in red and black, respectively.}
\end{figure}

	Thus, we may assume that \(z_iz_j\in E(P)\).
	Then \(j=i+1\) and, since \(P(z_{i-1},z_i)\) is a chord, \(z_{i-1} = y_3\).
	Here we have three cases,
	depending on whether (a) \(y_2 \notin V(P)\);
	(b) \(z_r = y_2\) and \(r < i-1\); or (c) \(z_r = y_2\) and \(r> i+1\).
	In case (a), we have \(z_1 = y_3\), \(z_2 = y_1\), and \(z_3 = y_4\).
        Therefore, \(P_1 = P(x_0,y_0) + y_0y_4y_3y_1y_2\) and \(P_2 = G - E(P_1)\) are two paths which decompose \(G\) (see Figure~\ref{fig:case13}).
	In case (b), we have \(z_1 = y_2\), \(z_2 = y_3\), \(z_3 = y_1\), and \(z_4 = y_4\).
	Since \(y_2y_3\in E(C)\), the subpath \(P(z_1,z_2)\) has length at least \(2\).
	Let \(z'\) be the neighbor of \(y_3\) in \(P(z_1,z_2)\).
	Therefore, \(P_1 = P(x_0,z_1) + y_2y_1y_4y_3z'\) and \(P_2 = G - E(P_1)\) are two paths which decompose \(G\) (see Figure~\ref{fig:case14}).
	In case (c),  we have \(z_1 = y_3\), \(z_2 = y_1\), \(z_3 = y_4\), and \(z_4 = y_2\).
	Since \(P\) contains at most three chords, at most one  between \(P(z_0,z_1)\) and \(P(z_3,z_4)\) is a chord.
	By symmetry, we can suppose that \(P(z_3, z_4)\) is not a chord,
	and let \(z'\) be the neighbor of \(y_2\) in \(P(z_3,z_4)\),
	Therefore, \(P_1 = P(x_0,z_0) + y_0y_4y_1y_3y_2z'\) and \(P_2 = G-E(P_1)\) are two paths which decompose \(G\) (see Figure~\ref{fig:case15}).
\begin{figure}[h]
	\centering
	\begin{subfigure}[b]{.5\linewidth}
	\centering\scalebox{.7}{\input{Figures/cycle-path6-before.tikz}}%
	\centering\scalebox{.7}{\input{Figures/cycle-path6-after.tikz}}
	\caption{}\label{fig:case14}
	\end{subfigure}%
	\begin{subfigure}[b]{.5\linewidth}
	\centering\scalebox{.7}{\input{Figures/cycle-path7-before.tikz}}%
	\centering\scalebox{.7}{\input{Figures/cycle-path7-after.tikz}}
	\caption{}\label{fig:case15}
	\end{subfigure}%

        \caption{The illustration on the left of figures (a) and (b) shows the path $P$ and the cycle $C$ of lenght \(5\) in red and black, respectively, while
                  the illustration on the right shows the paths $P_1$ and $P_2$ in red and black, respectively.}
\end{figure}
\end{proof}

Note that \(K_5^-\) can be decomposed into a path \(P\) and a cycle \(C\)
such that \(P\) contains precisely four chords of \(C\).
Therefore, Lemma~\ref{lemma:cycle+path}(ii) is tight.

\begin{corollary}\label{cor:triangles+path}\label{lemma:C4+path}
	Let \(G\) be a connected graph that can be decomposed
	into a non-empty graph \(H\) and \(k\) pairwise vertex-disjoint cycles of length 3 or 4.
	Then \(\pn(G)\leq \pn(H) + k\).
\end{corollary}

\begin{proof}
	The proof follows by induction on \(k\).
	Let \(\{H,C_1,\ldots,C_k\}\) be a decomposition as in the statement.
	If \(k=1\), then let \(\D'\) be a minimum path decomposition of \(H\),
	and let \(P\) be a path of \(\D'\) that intercepts \(C_1\).
	By Lemma~\ref{lemma:cycle+path}(ii), \(G' = P + C_1\) admits a decomposition
	into two paths, say \(P_1,P_2\).
	Therefore, \(\D = \D' - P + P_1+P_2\) is a path decomposition
	of \(G\) with size \(\pn(H) + 1\).
	Now suppose \(k>1\).
	Since \(G\) is connected, we have \(V(C_i)\cap V(H)\neq\emptyset\), for every \(i\in\{1,\ldots,k\}\).
	Thus, \(H'=H+C_k\) is connected.
	From the case \(k=1\) we have \(\pn(H') \leq \pn(H) + 1\).
	Now, note that \(\{H',C_1,\ldots,C_{k-1}\}\) is a decomposition of~\(G\) as in the statement.
	By the induction hypothesis, \(\pn(G) \leq \pn(H') + k-1\leq \pn(H)+k\).
\end{proof}

The next result is a version of Lemma~\ref{lemma:cycle+path} for \(K_5\) and \(K_5^-\).

\begin{lemma}\label{lemma:path+K5-orK5}
	If \(G\) is a connected graph that can be decomposed into a path
	and a copy of~\(K_5\) or a copy of \(K_5^-\),
	then \(\pn(G) = 3\).
\end{lemma}

\begin{proof}
	Let \(P\) be a path and \(H\) be a copy of \(K_5\) or \(K_5^-\),
	as in the statement.
	It is clear that~\(H\) admits a decomposition into a cycle \(C\),
	and a path or a cycle \(B\).
	If \(H\) is isomorphic to \(K_5^-\),
	then \(P\) contains at most one chord of \(C\).
	By Lemma~\ref{lemma:cycle+path}(i), \(P+C\) admits a decomposition into two paths,
	say \(P_1,P_2\).
	Therefore, \(\{P_1,P_2,B\}\) is a path decomposition of \(G\) with size~\(3\).
	If \(H\) is isomorphic to \(K_5\),
	then \(P\) contains no chord of \(C\) or \(B\).
	By Lemma~\ref{lemma:cycle+path}(i), \(P + C\) admits a decomposition into two paths,
	say \(P_1,P_2\),
	such that \(P_1\) contains exactly one edge of~\(C\).
	By Lemma~\ref{lemma:cycle+path}(i), \(P_1 + B\) admits a decomposition into two paths,
	say \(P_3,P_4\).
	Therefore, \(\{P_2,P_3,P_4\}\) is a path decomposition of \(G\) with size \(3\).
\end{proof}

The next result is an analogous version of Corollary~\ref{cor:triangles+path} for \(K_5\) and \(K_5^-\).
Its proof uses Lemma~\ref{lemma:path+K5-orK5} instead of Lemma~\ref{lemma:cycle+path}(ii).

\begin{corollary}\label{cor:extras-K5-}
	Let \(G\) be a connected graph
	that can be decomposed into a non-empty graph \(H\)
	and \(k\) pairwise vertex-disjoint copies of \(K_5\) or \(K_5^-\).
	Then \(\pn(G) \leq \pn(H) + 2k\).
\end{corollary}

The following result
shows that reducing subgraphs can absorb copies of \(K_3, K_5\) and~\(K^-_5\)
while keeping its reducing property.
From now on, given a graph \(G\) and a subgraph \(H\) of \(G\),
we denote by \(\mathcal{C}_3^H\) (resp. \(\mathcal{C}_5^H\))
the set of components of \(G-E(H)\) isomorphic to \(K_3\) (resp. \(K_5\) or \(K_5^-\)).
We denote by \(\mathcal{C}_3^\emptyset\) (resp. \(\mathcal{C}_5^\emptyset\))
the components of \(G\) isomorphic to \(K_3\) (resp. \(K_5\) or \(K_5^-\)).

\begin{lemma}\label{lemma:reducing+triangles+K5-}
	Let \(G\) be a graph such that \(\mathcal{C}_3^\emptyset=\mathcal{C}_5^\emptyset=\emptyset\),
	and let \(H\) be a \(r\)-reducing subgraph of \(G\).
	Then, \(H'=H+\bigcup_{C\in\mathcal{C}_3^H\cup\mathcal{C}_5^H}C\)
	is an \(\big(r+|\mathcal{C}_3^H|+2|\mathcal{C}_5^H|\big)\)-reducing subgraph of \(G\).
\end{lemma}
\begin{proof}
	Let \(G\), \(H\), and \(H'\) be as in the statement.
	Let \(I_H\) and \(I_{H'}\) be the isolated vertices of \(G-E(H)\) and \(G-E(H')\), respectively,
	and let \(r' = r + |\mathcal{C}_3^H|+ 2|\mathcal{C}_5^H|\).
	Since~\(H\) is an \(r\)-reducing subgraph of \(G\), we have \(\pn(H)\leq r\) and \(|I_H| \geq 2r\).
	Since \(\mathcal{C}_3^\emptyset=\mathcal{C}_5^\emptyset=\emptyset\),
	each component in \(\mathcal{C}_3^H\cup\mathcal{C}_5^H\) intersect \(H\).
	By applying Corollaries~\ref{cor:triangles+path} and~\ref{cor:extras-K5-} in each component of \(H'\),
	we have \(\pn(H') \leq \pn(H) + |\mathcal{C}_3^H|+ 2|\mathcal{C}_5^H| \leq r'\).
	Note that, if \(x\) is a vertex in \(V\big(\bigcup_{C\in\mathcal{C}_3^H\cup\mathcal{C}_5^H}C\big)\),
	then~\(x\) is not isolated in \(G-E(H)\), hence \(x\notin I_H\),
	but \(x\) is isolated in~\(G-E(H')\).
	Thus, we have
	\[|I_{H'}| = |I_H| + 3|\mathcal{C}_3^H|+ 5|\mathcal{C}_5^H| \geq 2(r + |\mathcal{C}_3^H|+ 2|\mathcal{C}_5^H|)\geq 2r'.\]
	Therefore, \(H'\) is an \(r'\)-reducing subgraph of \(G\).
\end{proof}

The next two results show that, although \(\pn(K_5^-)=3\),
any proper subdivision of \(K_5^-\) can be decomposed into two paths.

\begin{proposition}\label{prop:K5--simple-subdivision}
    If \(G\) is a graph obtained from \(K_5^-\) by a subdivision of one of its edges, then \(\pn(G) = 2\).
\end{proposition}

\begin{proof}

    Let \(G\) be as in the statement.
    Let \(u,v\) and \(x,y,z\) be the vertices of \(G\) with degree~\(3\) and~\(4\), respectively,
    and let \(w\) be the vertex of degree \(2\) of \(G\).
    There are two cases, depending on whether \(w\) has a neighbor of degree \(3\).
    First, suppose that \(w\) has a neighbor of degree \(3\).
    Since the vertices of degree \(3\) in \(K_5^-\) are are not adjacent,
    the other neighbor of \(w\) is a vertex of degree \(4\).
    We may suppose, without loss of generality, that \(N(w) = \{u,x\}\).
    In this case, \(\{wxyvzu,wuyzxv\}\) is a path decomposition of \(G\) with size \(2\) (see Figure~\ref{fig:sub-K5-a}).
    Thus, we may suppose that the two neighbors of \(w\) have degree \(4\).
    Suppose, without loss of generality, that \(N(w) = \{y, z\}\).
    In this case, \(\{wzxvyu, wyxuzv,\}\) is a path decomposition of \(G\) with size \(2\) (see Figure~\ref{fig:sub-K5-b}).
\end{proof}

\begin{figure}[h]
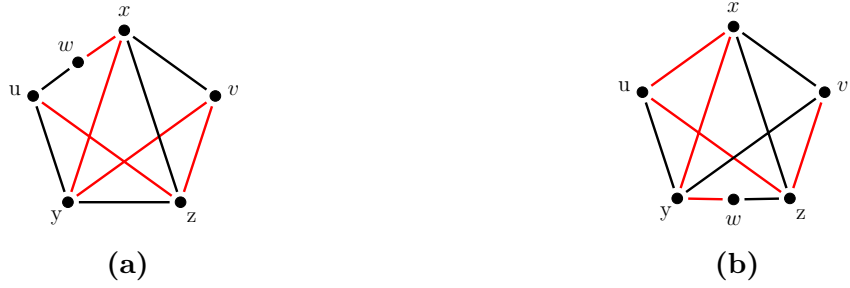

	\centering
	\begin{subfigure}[b]{.5\linewidth}
	\centering\scalebox{.7}{\input{Figures/sub-k5-a.tikz}}%
	\caption{}\label{fig:sub-K5-a}
	\end{subfigure}%
	\begin{subfigure}[b]{.5\linewidth}
	\centering\scalebox{.7}{\input{Figures/sub-k5-b.tikz}}%
	\caption{}\label{fig:sub-K5-b}
	\end{subfigure}%

	\caption{Decompositions into two paths of the graph obtained from \(K_5^-\) by subdividing one of its edges.}
\end{figure}

\begin{corollary}\label{cor:K5--subdivision}
    If \(G\) is a proper subdivision of \(K_5^-\), then \(\pn(G) = 2\).
\end{corollary}

\subsection{Liftings}

In this section, we present two results that allow us to obtain reducing subgraphs
in situations where we use liftings.
Let \(x, y, z\) be three vertices in a graph \(G\).
We say that \(\lift{xyz}\) is a \emph{valid lifting}
 if \(xy, yz \in E(G)\) and \(xz \notin E(G)\).
 In this case, the \emph{lifting} of \(xy,yz\) at \(y\) is the operation
of removing the edges \(xy\) and \(yz\) and adding the edge \(xz\),
which yields the graph \(G'=G-xy-yz+yz\).
Alternatively, we denote \(G'\) by \(G+\lift{xyz}\) and \(G\) by \(G' - \lift{xyz}\).
If \(y\) is a vertex of degree \(2\), the (only) lifting at \(y\) is called a \emph{suppression} of~\(y\).

\begin{lemma}\label{lemma:simple-lifting}
	Let \(G\) be a connected graph and \(P\subseteq G\) be a path.
	Let \(x,v,y\in V(G)\) such that \(\{x,v,y\}\cap V(P)\neq\emptyset\)
	and that \(\lift{zvy}\) is a valid lifting in \(G-E(P)\).
	If at least two vertices of \(G\) are isolated in \(G'=(G-E(P)) + \lift{xvy}\)
	and every component of~\(G'\) is a Gallai graph or isomorphic to \(K_3\) or to \(K_5^-\),
	then \(G\) contains a reducing subgraph.
\end{lemma}

\begin{proof}
	Let \(G, P, G'\) and \(x,v,y\) be as in the statement.
	Let \(C'\) be the component of \(G'\) containing \(xy\),
	and let \(C = C' - \lift{xvy}\).
	Put \(H = P+C\) and \(r = \lfloor |V(C')|/2\rfloor + 1\).
	If~\(C'\) is a Gallai graph, then \(\pn(C)\leq\pn(C')\leq \lfloor |V(C')|/2\rfloor\).
	If~\(C'\simeq K_5^-\),
	then~\(C\) is a subdivision of \(K_5^-\) and, by Corollary~\ref{cor:K5--subdivision},
	\(\pn(C) = 2  = \lfloor |V(C')|/2\rfloor\).
	In these cases, we have \(\pn(H) \leq \pn(C) + \pn(P) \leq \lfloor |V(C')|/2 \rfloor + 1 = r\).
	If~\(C'\simeq K_3\), then \(C\) is a cycle of length \(4\).
	Since \(\{x,v,y\}\cap V(P)\neq\emptyset\), the graph \(H\) is connected and,
	by Corollary~\ref{lemma:C4+path}, \(\pn(H) = 2 = 1 + \lfloor |V(C')|/2\rfloor\).
	Now, note that \(G - E(P) - E(C) = G' - E(C')=G''\).
	Since \(G'\) has at least two isolated vertices and \(C'\) is a component of \(G'\),
	the graph \(G''\) has at least \(2 + |V(C')| \geq 2r\) isolated vertices.
	Hence,~\(H\) is an \(r\)-reducing subgraph of~\(G\).
\end{proof}

\begin{lemma}\label{lemma:double-lifting}
	Let \(G\) be a connected graph and \(P\subseteq G\) be a path.
	Let \(x,v,y,a,u,b\in V(G)\) be such that \(v\neq u\),
	\(\{x,v,y\}\cap V(P)\neq\emptyset\) and
	\(\{a,u,b\}\cap V(P)\neq\emptyset\),
	and that \(\lift{xvy},\lift{aub}\) are valid liftings in \(G-E(P)\).
	Moreover, suppose that \(|\{a,b\}\cap N(v)|,|\{x,y\}\cap N(u)|\leq 1\).
	If at least two vertices of \(G\) are isolated in \(G'= (G-E(P)) + \lift{xvy} + \lift{aub}\),
	and every component of \(G'\) is a Gallai graph or isomorphic to \(K_3\) or to \(K_5^-\),
	then \(G\) contains a reducing subgraph.
\end{lemma}

\begin{proof}
	Let \(G, P, G'\) and \(x, v, y, a, u, b\) be as in the statement.
	First, suppose that there is only one component, say \(C'\), of \(G'\) containing the edges \(xy\) and \(ab\).
	Let \(C = C' - \lift{xvy} - \lift{aub}\) and \(H = P + C\).
	Since \(\{x,v,y\}\cap V(P)\neq\emptyset\) and \(\{a,u,b\}\cap V(P)\neq\emptyset\), the graph~\(H\) is connected.
	Put \(r = \lfloor |V(C')|/2\rfloor + 1\).
	If \(C'\) is a Gallai graph or \(C'\simeq K_5^-\), then \(\pn(C)\leq\pn(C')\leq \lfloor |V(C')|/2\rfloor\).
	Thus, \(\pn(H) \leq \pn(C) + \pn(P) \leq \lfloor |V(C')|/2 \rfloor + 1 = r\).
	If \(C'\simeq K_3\), then \(C\) is a cycle of length \(5\) and \(u,b,y,v,x,a\in V(C)\)
	(in this case, we have \(\{a,b\}\cap\{x,y\}\neq\emptyset\))
	and~\(P\) has at most three chords of \(C\), because \(|\{a,b\}\cap N(v)|,|\{x,y\}\cap N(u)|\leq 1\).
	Thus, by Lemma~\ref{lemma:cycle+path}(ii), \(\pn(H) =2 = 1 +\lfloor |V(C')|/2\rfloor= r\).
	In each of the cases above we have \(\pn(H) \leq  r\).
	Now, note that \(G - E(P) - E(C) = G' - E(C')=G''\).
	Since \(G'\) has two isolated vertices and \(C'\) is a component of \(G'\),
	the graph \(G''\) has \(2 + |V(C')| \geq 2r\) isolated vertices.
	Therefore,~\(H\) is an \(r\)-reducing subgraph of \(G\).

	Now, suppose that \(C'_1\) and \(C'_2\) are two distinct components of \(G'\) containing, respectively, \(xy\) and \(ab\).
	Let \(C_1 = C_1' - \lift{xvy}\), \(C_2 = C_2' -\lift{aub}\) and let \(H = P + C_1 + C_2\).
	Put \(r = \lfloor |V(C'_1)|/2\rfloor + \lfloor |V(C'_2)|/2\rfloor + 1\).
	We claim that \(H\) is an \(r\)-reducing subgraph of \(G\).
	Analogously to the case above,
	we have \(G - E(P) - E(C_1) - E(C_2) = G' - E(C'_1) - E(C'_2) = G''\),
	and \(G''\) has at least \(2 + |V(C'_1)| + |V(C'_2)|\geq 2r\) isolated vertices.
	Also,
	if \(C'_i\simeq K_5^-\) or \(C'_i\) is a Gallai graph, we have \(\pn(C_i)\leq \lfloor |V(C'_i)|/2\rfloor\), for \(i=1,2\).
	If \(C'_i\simeq K_3\), then \(C_i\) is a cycle of length \(4\).
	Put \(H' = P + C_1\).
	Since \(x,v,y\in V(C_1)\) and \(\{x,v,y\}\cap V(P)\neq\emptyset\), \(H'\) is connected.
	Again, we have \(\pn(H') \leq 1 + \lfloor |V(C'_1)|/2\rfloor\).
	(if \(C'_1\simeq K_3\), we use Lemma~\ref{lemma:cycle+path}(ii)).
	Now, note that \(H = H' + C_2\) is connected,
	because \(a,u,b\in V(C_2)\) and \(\{a,u,b\}\cap V(P)\neq\emptyset\).
	Analogously,
	if \(C'_2\simeq K_5^-\) or \(C'_2\) is a Gallai graph,
	\(\pn(H) \leq \pn(H') + \pn(C_2) \leq 1 + \lfloor |V(C'_1)|/2\rfloor + \lfloor |V(C'_2)|/2\rfloor \leq r\); and
	if \(C'_2\simeq K_3\), then by Lemma~\ref{lemma:cycle+path}(ii),
	we have \(\pn(H) \leq \pn(H') + 1 \leq 1 + \lfloor |V(C'_1)|/2\rfloor + \lfloor |V(C'_2)|/2\rfloor \leq r\).
	Therefore,~\(H\) is an \(r\)-reducing subgraph of \(G\).
\end{proof}

\section{Graphs with treewidth at most three}\label{sec:treewidth}

\newcommand{\note}[1]{\textcolor{magenta}{(#1)}}

In this section we verify Conjectures~\ref{conj:gallai} and~\ref{conj:hajos} for graphs with treewidth at most \(3\).
Let~\(k\) be a positive integer.
It is known that graphs with treewidth at most \(k\) are precisely the partial \(k\)-trees~\cite{Bodlaender1998}.
A graph \(G\) is a \emph{\(k\)-tree} if one of the following conditions holds:
(i) \(G\) is isomorphic to \(K_k\), or
(ii) \(G\) contains a vertex \(v\) such that \(G[N(v)]\simeq K_k\)
and \(G - v\) is a \(k\)-tree.
A vertex \(v\) of a \(k\)-tree~\(G\) is a \emph{terminal vertex} or, simply, a \emph{terminal}, of~\(G\)
if~\(d(v)=k\) or~\(G \simeq K_k\).
It is not hard to check that if \(G\) is a \(k\)-tree with at least \(k+1\) vertices and~\(v\) is a terminal of~\(G\),
then~\(G-v\) is a \(k\)-tree.
Therefore, it is possible to obtain a copy of \(K_k\) from any \(k\)-tree by a sequence of removals of terminal vertices.
The following facts show that every \(k\)-tree with at least~\(k+2\) vertices
contains at least two non-adjacent terminals.

\begin{fact}\label{fact:terminal-vertices}
	If \(G\) is a \(k\)-tree with at least \(k+2\) vertices,
	then every pair of terminals of \(G\) is non-adjacent.

\end{fact}

\begin{proof}
  	Let \(G\) be a \(k\)-tree with at least \(n\geq k+2\) vertices
  	and let \(v\) be a terminal of \(G\).
	Since \(v\) is a terminal, \(G'=G-v\) is a \(k\)-tree.
	Let \(u\) be a neighbor of \(v\).
	Since~\(G'\) has at least \(k+1\) vertices, \(d_{G'}(u)\geq k\),
	which implies \(d_G(u) \geq k+1\), hence \(u\) is not a terminal.
\end{proof}

\begin{fact}\label{fact:at-least-two-terminals}
	If \(G\) is a \(k\)-tree with at least \(k+2\) vertices,
	then \(G\) contains at least two terminals.
\end{fact}

\begin{proof}
  	Let \(G\) be a \(k\)-tree with at least \(n\geq k+2\) vertices.
	The proof follows by induction on \(n\).
	If \(n = k+2\), then \(G\) is isomorphic to \(K_{k+2}-e\), and the statement holds.
	Thus, suppose \(n > k+2\) and let \(v\) be a terminal of \(G\).
	By the definition of terminal, \(G'=G-v\) is a \(k\)-tree.
	Since every pair of neighbors of \(v\) is adjacent,
	by Fact~\ref{fact:terminal-vertices}, \(v\) is adjacent to at most one terminal of \(G'\).
	By the induction hypothesis, there is at least one terminal, say~\(u\), of \(G'\) which is
	not adjacent to \(v\).
	Therefore, \(u\) is a terminal of \(G'\).
\end{proof}

We say that a graph \(G\) is a \emph{partial \(k\)-tree} if it is a subgraph of a \(k\)-tree \(G^*\) (see~\cite{ArnborgEtAl1990,Bodlaender1998}).
In this case, we say that \(G^*\) is an \emph{underlying \(k\)-tree} of \(G\).
The next fact shows that there is an underlying \(k\)-tree \(G^*\) of \(G\)
such that \(G\) contains every terminal vertex of \(G^*\).
In fact, one can prove that if \(G\) has at least \(k\) vertices, then \(G\) is a spanning subgraph of a \(k\)-tree,
but, in this paper, we do not make use of this fact.

\begin{fact}\label{lem:k-tree-spanning}
  If \(G\) is a partial \(k\)-tree with at least \(k\) vertices,
  then there exists an underlying \(k\)-tree \(G^*\) of \(G\) such that
  \(G\) contains every terminal vertex of \(G^*\).
\end{fact}

\begin{proof}
	Let \(G\) be as in the statement
	and let \(G^*\) be an underlying \(k\)-tree of \(G\)
	with a minimum number of vertices.
	Suppose, for a contradiction, that there is a terminal \(v\) of~\(G^*\)
	such that \(v\notin V(G)\).
	Then \(G^*-v\) is an underlying \(k\)-tree of \(G\),
	a contradiction to the minimality of \(G^*\).
\end{proof}

Partial \(3\)-trees are also known by their forbidden minors characterization~\cite{ArnborgEtAl1990}.
Therefore, every subgraph of a partial \(3\)-tree is also a partial \(3\)-tree,
and the graph obtained from a partial \(3\)-tree \(G\) by
suppressing a vertex of degree \(2\) is also a partial \(3\)-tree.
Let \(G\) be a partial \(k\)-tree,
and let \(G^*\) be an underlying \(k\)-tree of \(G\).
We say that a vertex \(v\) of \(G\) is \emph{terminal} if \(v\) is a terminal vertex in \(G^*\).
Therefore, if \(G\) is a partial \(k\)-tree with at least \(k+2\) vertices,
then \(G\) contains at least two non-adjacent terminals.
The following fact about partial \(k\)-trees will be used often in the proof of Theorem~\ref{theorem:treewidth-most-three}.

\begin{fact}\label{fact:remove-terminal}
	Let \(G\) be a partial \(k\)-tree with at least \(k+1\) vertices,
	\(v\) be a terminal of~\(G\),
	and \(S\) be a set of edges joining vertices in \(N(v)\)
	such that \(S\cap E(G)=\emptyset\).
	Then \(G-v+S\) is a partial \(k\)-tree.
\end{fact}

\begin{proof}
	Let \(G\), \(k\), and \(S\) be as in the statement,
	and let \(G^*\) be the underlying \(k\)-tree of~\(G\).
	Note that \(G^*-v\) is an underlying \(k\)-tree of \(G-v\).
	Since \(N_{G^*}(v)\simeq K_k\), we have \(S\subseteq E(G^*)\).
	Thus, \(G^*\) is an underlying \(k\)-tree of \(G-v+S\).
\end{proof}

\subsection{Double centered $3$-trees}

In this section, we present and characterize a family of \(3\)-trees
that will be useful in the proof of Theorem~\ref{theorem:treewidth-most-three}.
We say that a \(3\)-tree \(G\) is \emph{double centered}
if there are two vertices \(a,b\in V(G)\) such that every terminal vertex of \(G\) different from \(a,b\) is adjacent to~\(a\) and~\(b\)
(see Figure~\ref{fig:example-centered-3-trees}).
In this case, we say that \(a,b\) are the \emph{centers} of \(G\).
The characterization of double-centered $3$-trees presented here is given in terms of \emph{graph joins}.
Given two vertex-disjoint graphs \(H_1\) and \(H_2\),
we say that a graph~\(G\) is the \emph{join} of \(H_1\) and \(H_2\), denoted by \(H_1\vee H_2\),
if \(G\) is the graph obtained from \(H_1 + H_2\) by joining every vertex of \(H_1\) to every vertex of \(H_2\),

i.e., \(V(G) = V(H_1) \cup V(H_2)\) and \(E(G) = E(H_1) \cup E(H_2) \cup \{uv\colon u\in V(H_1), v\in V(H_2)\}\).
\begin{proposition}\label{prop:double-centered-3-tree}
    Let \(G\) be a double centered \(3\)-tree with centers \(a\) and \(b\).
    There is a tree~\(T\) such that \(G\simeq T\vee K_2\),
    where \(V(T) = V(G) -a -b\).
    Moreover, the terminals of \(G - a - b\) are precisely the leafs of \(T\).
\end{proposition}
\begin{proof}
    Let \(G\) and \(a,b\) be as in the statement.
    The proof follows by induction on the number \(n\) of vertices of \(G\).
    If \(n=3\), then \(G\simeq K_3\), and the result follows with \(T \simeq K_1\).
    Now, suppose \(n \geq 4\), and let \(u\) be a terminal vertex of \(G\) different from \(a,b\).
    By the definition of double centered $3$-trees, we have \(a,b\in N(u)\).
    Let \(v\) be the neighbor of \(u\) different from \(a,b\),
    and let \(G' = G-u\).
    Since \(G\) is a \(3\)-tree, we have \(N(u)\simeq K_3\), hence \(a,b\in N(v)\).
    Note that that every terminal of \(G\), except for \(u\), is a terminal of \(G'\).
    The only possible terminal of \(G'\) (different from \(a,b\)) that is not a terminal of \(G\) is \(v\).
    Since \(N(u)\) induces a complete subgraph of \(G'\), \(v\) is adjacent to \(a,b\).
    Thus every terminal of \(G'\) different from \(a,b\) is adjacent to \(a\) and \(b\),
    hence \(G'\) is a double centered \(3\)-tree with \(n-1\) vertices.
    By the induction hypothesis, there is a tree~\(T'\) such that  \(G' \simeq T'\vee K_2\).
    Let \(T\) be the the tree obtained from \(T'\) by adding the vertex \(u\) and the edge \(uv\).
    Clearly, we have \(G = T\vee K_2\).
\end{proof}

  \begin{figure}[ht]
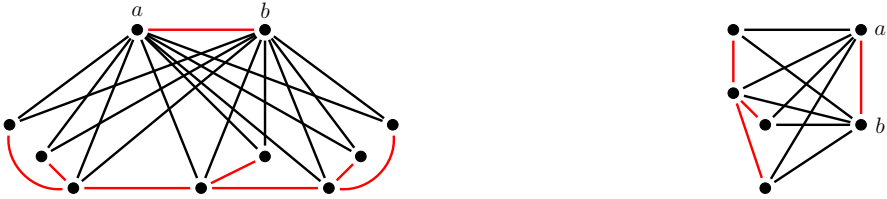

  \centering
  \begin{subfigure}[b]{.5\linewidth}
    \centering\scalebox{.7}{\input{Figures/example_centered_3_tree.tikz}}%
  \end{subfigure}%
  \begin{subfigure}[b]{.5\linewidth}
    \centering\scalebox{.7}{\input{Figures/example_centered_3_tree_b.tikz}}%
  \end{subfigure}%
  \caption{examples of double centered $3$-trees $T\vee K_2$ where $T$ is illustrated in red.}
  \label{fig:example-centered-3-trees}
\end{figure}

\subsection{Gallai's Conjecture for Graphs with Treewidth at most $3$}

In this section, we verify
Conjecture~\ref{conj:gallai} for graphs with treewidth at most \(3\).
In fact, we prove a slightly stronger result.
We prove that if \(G\) is a graph with \(n\) vertices and treewidth at most $3$,
then \(\pn(G) \leq \lfloor n/2\rfloor\), or \(G\) is isomorphic to \(K_3\) or to \(K_5^-\),
which is the graph obtained from \(K_5\) by the removal of one edge.
The proof consists of showing that a minimum counter-example for this statement
is either an odd graph or a subgraph of a double centered \(3\)-tree in which its centers have odd degree.
This implies contradictions to Theorem~\ref{thm:lovasz} and to Theorem~\ref{thm:pyber}, respectively.

Let $G$ be a graph, let $e \in E(G)$ be a cut-edge, and let $C_1$ and $C_2$ be the components of \(G-e\).
We say that $e$ is \emph{useful} if \(C_1\) and \(C_2\) have at least two vertices,
otherwise we say that \(e\) is \emph{useless}.
Now, we are able to prove our main result of this section.

\begin{theorem}\label{theorem:treewidth-most-three}
  Let $G$ be a connected partial $3$-tree with $n$ vertices.
  Then $\pn(G) \leq \floor{n / 2}$ or $G$ is isomorphic to $K_3$ or to $K_5^-$.
\end{theorem}

\begin{proof}
	Let \(G\) and \(n\) be as in the statement.
	Suppose, for a contradiction, that the statement does not hold,
	and let \(G\) be a counter-example for the statement that minimizes \(n\).
	It is not hard to check that \(n \geq 6\).

	\begin{claim}\label{claim:no-reducing}
		\(G\) contains no reducing subgraph.
	\end{claim}

	\begin{proof}
		Suppose, for a contradiction, that \(G\) contains an \(r\)-reducing subgraph \(H\).
		Since \(G\) is connected and \(n \geq 6\), we have \(\mathcal{C}^\emptyset_3 = \mathcal{C}^\emptyset_5 = \emptyset\).
		Thus, by Lemma~\ref{lemma:reducing+triangles+K5-},
		\(H'= H+\bigcup_{C\in\mathcal{C}_3^H\cup\mathcal{C}_5^H}C\)
		is an \(\big(r+|\mathcal{C}_3^H|+2|\mathcal{C}_5^H|\big)\)-reducing subgraph of \(G\).
		Moreover, no component of \(G-E(H')\) is isomorphic to \(K_3\) or to \(K_5^-\),
		thus \(H'\) is a Gallai graph.
		By Lemma~\ref{lemma:reducing}, \(G\) is a Gallai graph.
	\end{proof}

	\begin{claim}\label{cl:twm3:not-contain-useful-bridge}
	$G$ contains no useful cut-edge.
	\end{claim}

	\begin{proof}
	Suppose, for a contradiction, that $e = v_1v_2$ is a useful cut-edge in $G$.
	For \(i=1,2\), let $G_i$ be the component of $G - e$
	that contains $v_i$, and let $G'_i = G_i + e$.
	Note that~\(G'_i\) is a subgraph of \(G\),
	hence \(G'_i\) is a partial \(k\)-tree.
        Let \(n_i = |V(G'_i)|\).
	Since $e$ is a useful cut-edge, $n_i < n$,
	and since $G'_i$ has a vertex with degree one, \(G'_i\) is not isomorphic to $K_3$ or to $K_5^-$.
	Thus, by the minimality of $G$, the graph $G'_i$ is a Gallai graph, for \(i=1,2\).
	Let $\D_i$ be a path decomposition of $G'_i$ such that $|\D_i| \leq \floor{n_i / 2}$,
	and let \(P_i\) be the path in~\(\D_i\) that contains the edge \(e\).
	Put \(P = P_1 + P_2\), and \(\D = \D_1-P_1 + \D_2-P_2 + P\).
	Note that the vertices~\(v_1\) and~\(v_2\) are the only vertices of~\(G\) in both~\(G'_1\) and~\(G'_2\),
	hence \(n_1+n_2 =n+2\).
	Therefore, we have
	\(|\D| = |\D_1| + |\D_2| - 1 \leq \lfloor n_1/2\rfloor + \lfloor n_2/2\rfloor -1 \leq \lfloor(n_1+n_2-2)/2\rfloor = \lfloor n/2\rfloor\).
  	\end{proof}

  	\begin{claim}\label{claim:liftable-vertices-d2}
  		If \(v\) is a vertex of degree \(2\) in \(G\) and \(N(v) =\{x,y\}\),
  		then \(xy\in E(G)\).
  	\end{claim}

  	\begin{proof}
                Suppose for a contradiction that \(xy \notin E(G)\), and note that \(G' = G + \widehat{xvy}\) is a partial \(3\)-tree.
                Since \(n \geq 6\) either \(G'\) is a Gallai graph or \(G' \simeq K_5^-\).
		If~\(G'\) is a Gallai graph, then \(\pn(G)\leq\pn(G')\leq \lfloor (n-1)/2\rfloor\).
		If~\(G'\simeq K_5^-\),
		then~\(G\) is a subdivision of \(K_5^-\) and, by Corollary~\ref{cor:K5--subdivision},
		\(\pn(G) = 2 =\lfloor (n-1)/2\rfloor\).
  		In each case we obtained a contradiction to the minimality of~\(G\).
  	\end{proof}

	\begin{claim}\label{cl:twm3:2-edge-cut}
		If \(\{ab,cd\}\subset E(G)\) is an edge-cut,
		then \(\{a,b,c,d\}\) induces a cycle in \(G\).
	\end{claim}

	\begin{proof}
	Let \(G' = G - ab-cd\).
	If \(G'\) contains more than two components, then both~\(ab\) and~\(cd\)
	are cut-edges, hence \(\{ab,cd\}\) is not an edge-cut (a minimal edge separator).
	Thus, let \(G_1',G_2'\) be the two components of \(G'\).
	Suppose, without loss of generality, that \(a,c\in V(G_1')\) and \(b,d\in V(G_2')\).
	Thus, we have \(ab, ad, cb, cd\notin E(G)\).
	If~\(\{a,b,c,d\}\) does not induce a cycle in \(G\),
	then we can suppose that \(b\neq d\) and~\(bd\notin E(G)\).
	Let  \(G_2 = G_2'+bd\).
	By the minimality of \(G\), \(G_2\) is a Gallai graph
	or is isomorphic to~\(K_3\) or to~\(K_5^-\).

	If \(G_2\simeq K_3\), then let \(x\) be the vertex of \(G_2'\) different from \(b\) and \(d\).
        Note that in \(G\), the vertex \(x\) has degree two and has non-adjacent neighbors,
	a contradiction to Claim~\ref{claim:liftable-vertices-d2}.
	Thus, \(G_2\) is a Gallai graph or \(G_2\simeq K_5^-\).
	Let \(G_1 = G_1' + ab + cd\), and let \(P\) be a path in \(G_1\) joining~\(b\) to~\(d\).
	Let \(H = G_2' + P\), and note that \(H\) is a proper subdivision of \(G_2\) (see Figure~\ref{fig:2-edge-cut}).
	Let \(r = \lfloor |V(G_2')|/2\rfloor\).
	Note that \(G-E(H)\) has at least \(|V(G_2')|\geq 2r\) isolated vertices.
	If \(G_2\) is a Gallai graph, then \(\pn(H)\leq\pn(G_2)\leq\lfloor |V(G_2')|/2\rfloor = r\).
	If \(G_2\simeq K_5^-\), then by Proposition~\ref{prop:K5--simple-subdivision},
	we have \(\pn(H)\leq 2 = \lfloor |V(G_2')|/2\rfloor = r\).
	Thus, \(H\) is an \(r\)-reducing subgraph,
	a contradiction to Claim~\ref{claim:no-reducing}.
        \begin{figure}[ht]
          \centering
          \scalebox{.7}{\input{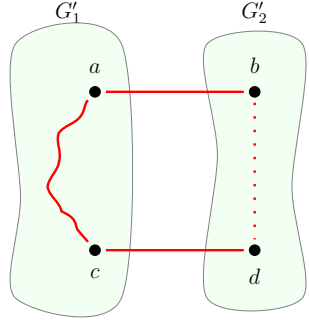}}%
          \caption{the two components \(G_1'\) and \(G_2'\) of \(G'\) are illustrated by the regions,
          					while the path \(P\) is colored red.}
	\label{fig:2-edge-cut}
        \end{figure}
	\end{proof}

        Let \(G^*\) be an underlying \(3\)-tree of \(G\).
        Since \(G\subseteq G^*\) and \(n\geq 6\), \(G^*\) has at least \(6\) vertices.
	By Facts~\ref{fact:terminal-vertices} and~\ref{fact:at-least-two-terminals},
	$G^*$ has at least two non-adjacent terminals, say~\(u\) and~\(v\).
	Recall that \(d_{G^*}(u) = d_{G^*}(v) = 3\).
	Thus, \(d_G(u),d_G(v)\leq 3\).
	The next claim shows that these terminals must have degree precisely \(3\).

  	\begin{claim}\label{claim:every-terminal-degree-3}
  		Every terminal of \(G\) has degree \(3\).
  	\end{claim}

  	\begin{proof}
  		Let \(u\) and \(v\) be two (non-adjacent) terminals of \(G\),
  		where \(d(u)\geq d(v)\).
  		Suppose for a contradiction that \(d(v) \leq 2\).
  		If \(d(u) = d(v) = 1\), then any path joining \(u\) and \(v\) is a reducing path,
  		a contradiction to Claim~\ref{claim:no-reducing}.
Suppose that \(d(u) = 2\), and let \(N(u) = \{a, b\}\).
  		Suppose also that \(u\) and \(v\) have at most one neighbor in common,
  		and let \(P'\) be a shortest path joining \(u\) and \(v\).
                Suppose, without loss o generality, that \(a\) is the neighbor of \(u\) in~\(P'\),
  		and let \(P\) be the graph obtained from \(P'\) by the addition of the edges incident to~\(u\) or~\(v\)
  		that are not in \(P'\) (see Figure~\ref{fig:n1}).
                Since \(P'\) is a shortest path and \(u\) and \(v\) have at most one neighbor in common,
                \(P\) is a path, and
  		since \(u\) and \(v\) are isolated in \(G-E(P)\),~\(P\) is a reducing path,
  		a contradiction to Claim~\ref{claim:no-reducing}.
  		Thus, we may suppose that \(u\) and \(v\) have precisely two neighbors, say \(a\) and \(b\), in common.
		Put \(G' = G -u-v\).
		By Claim~\ref{claim:liftable-vertices-d2}, we have \(ab\in E(G)\),
		hence \(G'\) is connected and \(|V(G')| \geq 4\).
		Thus, \(G'\not\simeq K_3\).
		By Fact~\ref{fact:remove-terminal},~\(G'\) is a partial \(3\)-tree.
		If \(G'\simeq K_5^-\),
		then \(G'' = G'-\lift{aub} =G'-ab+au+ub\) is a proper subdivision of \(K_5^-\),
		and by Corollary~\ref{cor:K5--subdivision} we have \(\pn(G'')=2\).
		Let \(T\) be the triangle with the edges \(va,vb,ab\),
		and note that \(G = G''+T\).
		By Corollary~\ref{cor:triangles+path}, \(\pn(G)\leq 3=\lfloor 7/2\rfloor = \lfloor n/2\rfloor\).
		Thus, we may suppose that \(G'\) is a Gallai graph and, therefore, \(\pn(G') \leq \lfloor (n-2)/2\rfloor\).
		By Corollary~\ref{cor:triangles+path}, \(\pn(G)\leq\pn(G') + 1 \leq\lfloor n/2\rfloor\).

		\begin{figure}[ht]
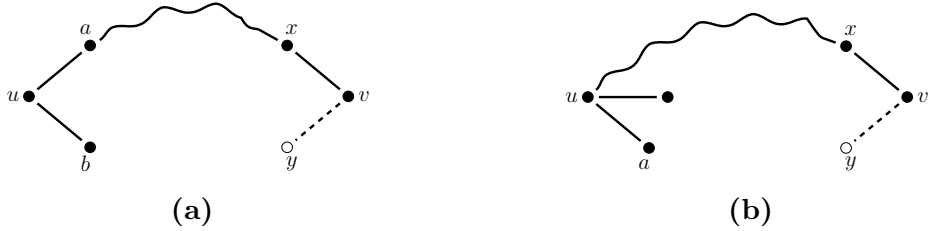

			\centering
			\begin{subfigure}[b]{.45\linewidth}
				\centering\scalebox{.7}{\input{Figures/n1.tikz}}%
				\caption{}\label{fig:n1}
			\end{subfigure}
			\begin{subfigure}[b]{.45\linewidth}
				\centering\scalebox{.7}{\input{Figures/n1-2.tikz}}%
				\caption{}\label{fig:n1-2}
			\end{subfigure}
			\caption{the path \(P\) joining \(u\) and \(v\),
				that can be extended to a path by adding the possible remaining edge \(vb\).}
		\end{figure}

		Thus, we can suppose that \(d(u) = 3\), and let \(N(u) = \{a, b, c\}\).
		First, we claim that that \(G-u\) is connected.
  		Indeed if \(G-u\) is not connected, then \(u\) is incident to a cut-edge, say~\(ua\).
  		By Claim~\ref{cl:twm3:not-contain-useful-bridge}, \(ua\) is a useless cut-edge, hence \(d(a) = 1\).
  		Since \(u\) and \(v\) are not adjacent, we have \(v \neq a\).
 		Again, if~\(P'\) is a minimum path joining \(a\) to \(v\),
 		and \(P\) is the path obtained from \(P'\) by adding the possible remaining edge
 		incident to \(v\) (see Figure~\ref{fig:n1-2}),
 		then \(a\) and \(v\) are isolated in \(G-E(P)\),
  		hence \(P\) is a reducing path, a contradiction to Claim~\ref{claim:no-reducing}.
  		Now, we claim that \(N(u)\) induces a clique in \(G\).
  		Indeed, if \(N(u)\) does not induces a clique in \(G\),
  		then there exists a pair of non-adjacent vertices in $N(u)$, say~$b$ and~$c$.
		Let~\(P'\) be a minimum path in \(G - u\) joining~\(a\) and~\(v\).
                Since~\(P'\) is a minimum path, it contains only one neighbor of \(v\), say \(x\).
  		If \(d(v) = 2\), then let \(y\) be the neighbor of \(v\) different from \(x\), and put \(P = P'+vy + au\).
  		If \(d(v) = 1\) put \(P = P'+au\) (see Figure~\ref{fig:n4}).
  		Let~\(G' = (G -E(P)) +\lift{buc}\).
  		By Fact~\ref{fact:remove-terminal}, \(G'\) is a partial \(3\)-tree,
  		and by the minimality of~\(G\), every component of \(G'\) is a Gallai graph, or is isomorphic to \(K_3\) or to \(K_5^-\).
  		Moreover, \(u\) and \(v\) are isolated in \(G'\).
  		Therefore, by Lemma~\ref{lemma:simple-lifting},
  		\(G\) contains a reducing subgraph, a contradiction to Claim~\ref{claim:no-reducing}.
                Therefore, we may assume that \(N(u)\) induces a clique in \(G\).

		\begin{figure}[ht]
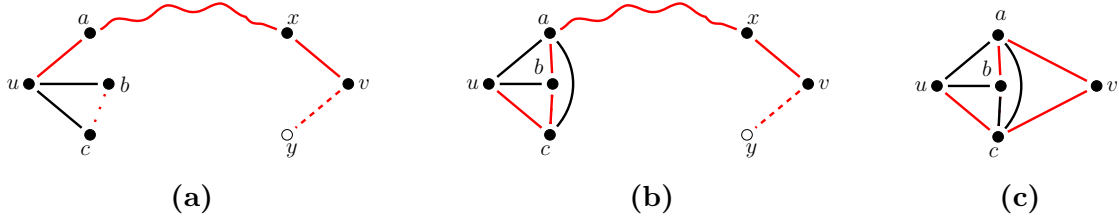

		\centering
			\begin{subfigure}[b]{.35\linewidth}
			\centering\scalebox{.7}{\input{Figures/n4.tikz}}%
			\caption{}\label{fig:n4}
			\end{subfigure}%
      \hspace{ 0.3cm}
			\begin{subfigure}[b]{.35\linewidth}
			\centering\scalebox{.7}{\input{Figures/n5.tikz}}%
			\caption{}\label{fig:n5}
			\end{subfigure}%
      \hspace{ 0.3cm}
			\begin{subfigure}[b]{.2\linewidth}
			\centering\scalebox{.7}{\input{Figures/n6.tikz}}%
			\caption{}\label{fig:n6}
			\end{subfigure}%
		\caption{Figures (a)-(c) show the path \(P\), in red,
				which its removal allows us to apply a lifting and then Lemma~\ref{lemma:simple-lifting},
				to obtain a reducing subgraph of \(G\).}
		\end{figure}

  		Let $P'$ be a minimum path from \(N(u)\) to $v$ in $G - u$.
  		Suppose, without loss of generality, that \(P'\) is a path from \(a\) to \(v\).
                Since \(P'\) is minimum, it does not contain \(b\) and \(c\).
  		Let \(x\) be the neighbor of \(v\) in~\(P'\).
		Suppose that $|N(u) \cap N(v)| \leq 1$.
  		If \(d(v) = 2\), then let \(y\) be the neighbor of \(v\) different from~\(x\), and put \(P = P' + ab + bc + cu + vy\).
  		If \(d(v) = 1\), then put \(P = P' + ab + bc + cu\) (see Figure~\ref{fig:n5}).
		Put \(G' = (G - E(P)) + \lift{aub}\).
  		Analogously to the case above,
  		\(G\) contains a reducing subgraph, a contradiction to Claim~\ref{claim:no-reducing}.
  		Thus, we may suppose that $|N(u) \cap N(v)| =2$.
  		Suppose that \(a,c\in N(v)\).
  		Let \(P = ucvab\) (see Figure~\ref{fig:n6}) and let \(G' = (G - E(P)) + \lift{aub}\).
  		Analogously to the cases above,
		\(G\) contains a reducing subgraph, a contradiction to Claim~\ref{claim:no-reducing}.
  	\end{proof}

In what follows, we present three properties of \(G\).
Let \(u\) and \(v\) be two non-adjacent terminal vertices.
Property~\ref{step:1} states that \(u\) and \(v\) have at least two neighbors in common,
and Property~\ref{step:2} states that \(u\) and \(v\) do not have (all) three neighbors in common.
Finally, Property~\ref{step:3} states that every common neighbor of two terminal has odd degree,
which enables us to characterize the minimum counter-example,
and show that it satisfies the statement, obtaining a contradiction.

\begin{property}\label{step:1}
There are no two terminals with at most one neighbor in common.
\end{property}

\begin{proof}
Suppose, for a contradiction, that there are two terminals, say \(u\) and \(v\),
such that \(|N(u)\cap N(v)| \leq 1\),
and let $G_u = G - u$, $G_v = G - v$ and $G_{uv} = G - u - v$.

	\begin{claim}\label{claim:twm3:Guv_disconnected}
	The graph $G_{uv}$  is disconnected.
	\end{claim}

	\begin{proof}
	Suppose for a contradiction that $G_{uv}$ is connected.
	First, suppose that $N(u)$ and $N(v)$ are not cliques in $G$.
	Thus, there are vertices $a,b\in N(u)$ and $x,y\in N(v)$ such that $ab,xy \notin E(G)$.
	Let $c$ and $z$ be the remaining vertices in $N(u)$ and $N(v)$, respectively.
	Since $G_{uv}$ is connected, there is a path $P'$ from $c$ to $z$ in $G_{uv}$.
	Let \(P = P' + uc + vz\) (see Figure~\ref{fig:g4}),
	and put $G' = (G - E(P)) + \lift{aub} + \lift{xvy}$.
	It is clear that \(u\) and \(v\) are isolated in \(G'\).
	Moreover, we have \(|\{a,b\}\cap N(v)|,|\{x,y\}\cap N(u)|\leq 1\), because \(|N(u)\cap N(v)| \leq 1\).
	Therefore, by Lemma~\ref{lemma:double-lifting},
	\(G\) contains a reducing subgraph,
	a contradiction to Claim~\ref{claim:no-reducing}.

        \begin{figure}[ht]
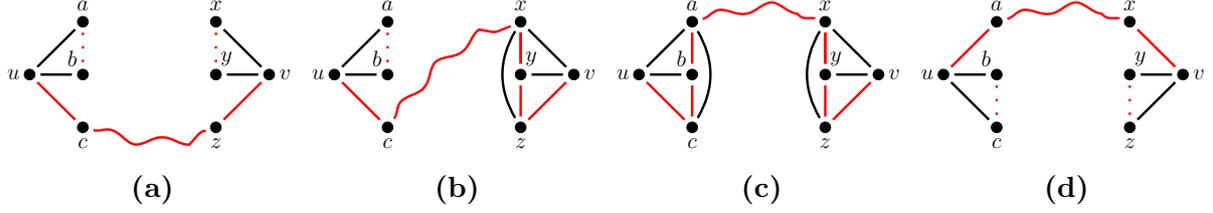

          \centering
          \hspace{ 0cm}
          \begin{subfigure}[b]{.25\linewidth}
            \centering\scalebox{.7}{\input{Figures/g4.tikz}}%
            \caption{}\label{fig:g4}
          \end{subfigure}%
          \begin{subfigure}[b]{.25\linewidth}
            \centering\scalebox{.7}{\input{Figures/g5.tikz}}%
            \caption{}\label{fig:g5}
          \end{subfigure}%
          \begin{subfigure}[b]{.25\linewidth}
            \centering\scalebox{.7}{\input{Figures/g6.tikz}}%
            \caption{}\label{fig:g6}
          \end{subfigure}%
          \begin{subfigure}[b]{.25\linewidth}
            \centering\scalebox{.7}{\input{Figures/g7.tikz}}%
            \caption{}\label{fig:g7}
          \end{subfigure}%
		\caption{Figures (a)-(d) show the path \(P\), in red,
				which its removal allows us to apply a lifting and then Lemma~\ref{lemma:double-lifting},
				to obtain a reducing subgraph of \(G\).}
        \end{figure}

	Thus, we can assume with without loss of generality that $N(v)$ is a clique.
	Now, suppose that $N(u)$ is not a clique in $G$,
        and let $a$ and $b$ be two nonadjacent vertices in~$N(u)$.
	Let $c$ be the remaining vertex in $N(u)$,
        and let $P'$ be a minimum path in~$G_{uv}$ from $c$ to a vertex, say $x$, in $N(v)$.
	Let $y$ and $z$ be the remaining vertices of $N(v)$,
	and let \(P = uc + P' + xy + yz + zv\) (see Figure~\ref{fig:g5}).
	Clearly, $P$ is a path in $G$.
	Now let $G' = (G - E(P)) +\lift{aub}+\lift{xvy}$.
	Analogously to the case above,
	\(G\) contains a reducing subgraph,
	a contradiction to Claim~\ref{claim:no-reducing}.

        Finally, we can assume that \(N(u)\) is a clique,
        and let $P'$ be a minimum path from a vertex in  $N(u)$ to a vertex in $N(v)$.
        We may suppose, without loss of generality, that \(P'\) joins \(a\) and \(x\).
	Put \(P = uc + cb + ba + P' + xy + yz + zv\) (see Figure~\ref{fig:g6})
	and \(G' = (G - E(P)) +\lift{aub}+\lift{xvy}\).
	Analogously to the cases above,
	\(G\) contains a reducing subgraph,
	a contradiction to Claim~\ref{claim:no-reducing}.
  \end{proof}

	From now on, we fix a minimum path $P'$ in $G$ joining a vertex in $N(u)$ to a vertex in $N(v)$.
        We may suppose, without loss of generality, that \(P'\) joins \(a\) and \(x\),
	and put \(P = ua + P' + xv\).
	Let $b$ and $c$ be the remaining vertices of $N(u)$, and let $y$ and $z$ be the remaining vertices of $N(v)$.
	We claim that at least one edge between~\(bc\) and~\(yz\) belong to~\(E(G)\).
	Indeed, suppose that the edges $bc,yz\notin E(G)$ (see Figure~\ref{fig:g7}).
	Thus, let $G' = (G - E(P)) + \lift{buc}+\lift{yvz}$.
	Since \(|N(u)\cap N(v)| \leq 1\), we have \(|\{a,b\}\cap N(v)|,|\{x,y\}\cap N(u)|\leq 1\).
	Analogously to the proof of Claim~\ref{claim:twm3:Guv_disconnected},
	\(G\) contains a reducing subgraph, a contradiction to Claim~\ref{claim:no-reducing}.
	Therefore, $G$ contains at least one of the edges $bc$ or $yz$.
	Assume, without loss of generality, that $G$ contains the edge $yz$.

	\begin{claim}\label{claim:twm3:v-not-cut-vertex}
	$u$ and \(v\) are not cut-vertices.
	\end{claim}

	\begin{proof}
	Suppose for a contradiction that $v$ is a cut-vertex.
	Since $d(v) = 3$, at least one of the edges incident to $v$ must be a cut-edge.
	Since \(yz\in E(G)\), the edges \(vy,vz\) are not cut-edges,
        hence $xv$ is a cut-edge.
	Since \(P' + au\) joins \(x\) to \(u\), the component of \(G-v\) that contains \(x\) is not trivial.
	Therefore, $xv$ is a useful cut-edge, a contradiction to Claim~\ref{cl:twm3:not-contain-useful-bridge}.
	
        To prove that \(u\) is not a cut vertex,
        we first show that \(b\) and \(c\) have degree at least~\(2\).
	If \(d(b) =1\), then \(ab, bc \notin E(G)\).
	Since \(v\) is not a cut-vertex,
	there is a minimum path \(Q'\) in \(G-v\) joining a vertex in $\{y,z\}$ to a vertex in \(\{a, c\}\).
	We may assume, without loss of generality, that \(Q'\) is a path joining \(c\) and \(y\).
	By the minimality of \(Q'\) we have \(u \notin V(Q')\).
        Suppose for a contradiction that \(xy \notin E(G)\).
	Let \(Q = Q' + uc + yz + zv\) (see Figure~\ref{fig:m8}),
        and let \(G' = (G - E(Q)) +\lift{aub} +\lift{xvy}\).
        Since \(|N(u) \cap N(v)| \leq 1|\), we have \(|\{a, b\} \cap N(v)| \leq 1\) and \(|\{x, y\} \cap N(u)| \leq 1\).
        By Lemma~\ref{lemma:double-lifting}, \(G\) contains a reducing subgraph, a contradiction to Claim~\ref{claim:no-reducing}.
        Thus, we can suppose that \(xy \in E(G)\).
        In this case, let \(Q = P' + ua + xy + yz + zv\) (see Figure~\ref{fig:m9}),
	and let \(G' = (G - E(Q)) +\lift{buc} +\lift{xvy}\).
        Again, by Lemma~\ref{lemma:double-lifting}, \(G\) contains a reducing subgraph, a contradiction to Claim~\ref{claim:no-reducing}.
	By symmetry we have \(d(c) > 1\).

        \begin{figure}[ht]
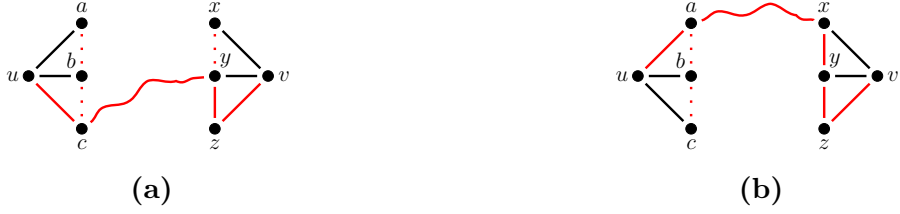

          \centering
          \begin{subfigure}[b]{.5\linewidth}
            \centering\scalebox{.7}{\input{Figures/m8.tikz}}%
            \caption{}\label{fig:m8}
          \end{subfigure}%
          \begin{subfigure}[b]{.5\linewidth}
            \centering\scalebox{.7}{\input{Figures/m9.tikz}}%
            \caption{}\label{fig:m9}
          \end{subfigure}%
		\caption{Figures (a) and (b) show the path \(P\), in red,
				which its removal allows us to apply a lifting and then Lemma~\ref{lemma:double-lifting},
				to obtain a reducing subgraph of \(G\).}
        \end{figure}

	Now, suppose that \(u\) is a cut-vertex.
	Since the vertex \(d(u) = 3\), at least one of the edges incident to $u$ is a cut-edge.
	Since \(a\) is adjacent to \(u\) and is an end-vertex of \(P'\), we have \(d(a) > 1\).
	Thus, since \(d(b) \geq 2\) and \(d(c) \geq 2\), if~\(e\) is a cut edge incident to~\(u\), then~\(e\) is useful,
	a contradiction to Claim~\ref{cl:twm3:not-contain-useful-bridge}.
	\end{proof}

Since \(P'\) is a path joining \(a\) and \(x\) in \(G_{uv}\),
the vertices \(a\) and \(x\) are in the same component, say \(C_1\), of \(G_{uv}\).
Since \(u\) is not a cut-vertex of \(G\), there is a component, say~\(C_2\neq C_1\), of \(G_{uv}\)
containing \(y\) and \(z\).
Note that \(C_1\) and \(C_2\) are the only components of~\(G_{uv}\),
since \(G_u\) is connected, \(G_{uv}\) disconnected, and \(y\) and \(z\) are adjacent in \(G_{uv}\).
Since~\(v\) is not a cut-vertex of \(G\), \(u\) must have a neighbor in \(C_2\).
Suppose, without loss of generality, that \(c \in V(C_2)\).
Now, suppose that \(b \in V(C_2)\), hence \(\{ua,vx\}\) is an edge-cut of \(G\) (see Figure~\ref{fig:m10}).
Since \(uv \notin E(G)\), \(\{u,a,v,x\}\) does not induce a cycle in \(G\),
a contradiction to Claim~\ref{cl:twm3:2-edge-cut}.
Thus, we can suppose that \(b\in V(C_1)\).
Hence, \(\{uc,vx\}\) is an edge-cut (see Figure~\ref{fig:m11}).
By hypothesis, the only possible neighbor in common between \(u\) and \(v\) is \(a\).
Thus, \(c\notin N(v)\), which means that \(cv\notin E(G)\).
Therefore, \(\{u,c,x,v\}\) does not induce a cycle in \(G\),
a contradiction to Claim~\ref{cl:twm3:2-edge-cut}.
\begin{figure}[ht]
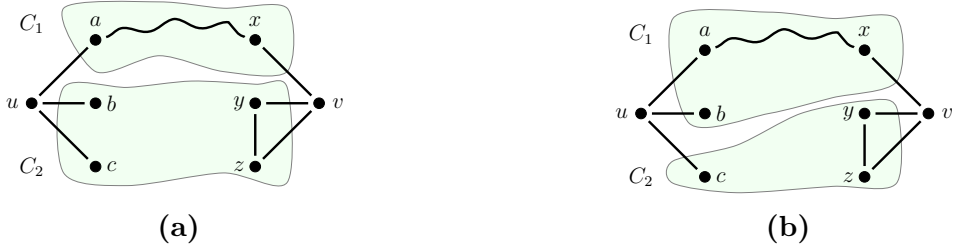

  \centering
  \begin{subfigure}[b]{.5\linewidth}
    \centering\scalebox{.7}{\input{Figures/m10.tikz}}%
    \caption{}\label{fig:m10}
  \end{subfigure}%
  \begin{subfigure}[b]{.5\linewidth}
    \centering\scalebox{.7}{\input{Figures/m11.tikz}}%
    \caption{}\label{fig:m11}
  \end{subfigure}%
  \caption{the possible components of \(G_{uv}\).}
\end{figure}
\end{proof}

\begin{property}\label{step:2}
There are no two terminals with (all) three neighbors in common.
\end{property}

\begin{proof}
	Suppose for a contradiction that \(u\) and \(v\) are two non-adjacent terminals
	such that \(N(u) = N(v) = \{a, b, c\}\).

        \begin{claim}
          \(\{a, b, c\}\) induces a clique.
        \end{claim}

	\begin{proof}
		Suppose for a contradiction that \(\{a,b,c\}\) does not induce a clique.
		We can suppose, without loss of generality, that \(ab\notin E(G)\).
		We claim that \(d(a)\) is odd.
		Indeed, suppose \(d(a)\) is even, and let \(G' = G - u - v + ab\).
		By Fact~\ref{fact:remove-terminal}, \(G'\) is a partial 3-tree.
		It is clear that \(d_{G'}(a)\) is odd.
		Thus, the component \(C'\) of \(G'\) that contains \(ab\) is not isomorphic to \(K_3\).
		Thus, by the minimality of \(G\), \(C'\) is a Gallai graph or \(C'\simeq K_5^-\).
		Now, put \(C = C' - \lift{avb}\), and note that \(\pn(C)\leq\lfloor|V(C')|/2\rfloor\)
		(if \(C'\simeq K_5^-\), then we use Proposition~\ref{prop:K5--simple-subdivision}).
		Let \(\D\) be a minimum path decomposition of \(C\).
		Since \(d_{G'}(a)=d_{C'}(a)=d_C(a)\) is odd, at least one path \(P\in\D\) has \(a\) as an end-vertex.
		Note that~\(Q = P + au\) and~\(R = bucv\) are paths of~\(G\) (see Figure~\ref{fig:m12}).
		Therefore, we have \(\pn(C+au) = \pn(C)\).
		Let \(H = C + au + R\), we have \(\pn(H) \leq \pn(C+au) + 1\leq \lfloor |V(C')|/2\rfloor + 1\).
		Finally, note that \(u,v\) and every vertex in \(V(C')\) is isolated in \(G-E(H)\),
		hence \(H\) is a reducing subgraph of \(G\), a contradiction to Claim~\ref{claim:no-reducing}.
                Therefore, \(d(a)\) is odd.
                By symmetry, \(d(b)\) is odd.

                \begin{figure}[ht]
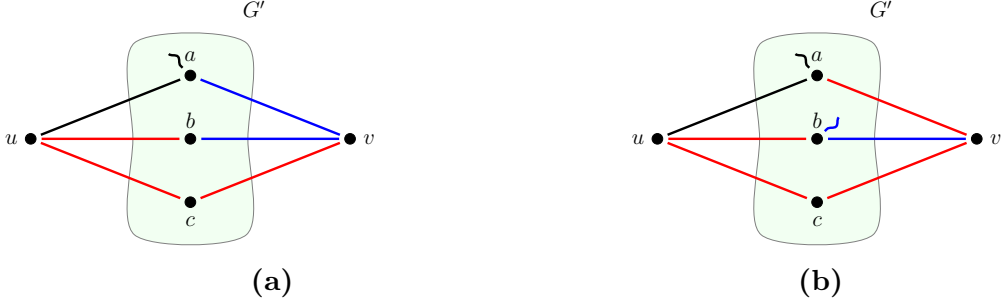

                  \centering
                \begin{subfigure}[b]{.45\linewidth}
                  \scalebox{.7}{\input{Figures/m12.tikz}}%
                  \caption{} \label{fig:m12}
		\end{subfigure}%
                  \begin{subfigure}[b]{.45\linewidth}
                    \centering
                    \scalebox{.7}{\input{Figures/m13-a-new.tikz}}%
                  \caption{} \label{fig:m13-a}
                  \end{subfigure}%

                  \label{fig:m12-old}
                  \caption{the paths \(Q\) and \(R\) are illustrated, respectively, in black and red.}
                \end{figure}

		Now, let \(G' = G - u -v\).
		It is not hard to check that if \(G'\) is disconnected,
		then for some \(x\in\{a,b,c\}\) the set \(\{ux,xv\}\)
		is a \(2\)-edge-cut, but since \(uv\notin E(G)\),
		\(\{u,x,v\}\) does not induce a cycle in \(G\).
                Since \(G'\) is a subgraph of \(G\), \(G'\) is a partial 3-tree,
                and by the minimality of \(G\),~\(G'\) is a Gallai graph or isomorphic to \(K_3\) or to \(K_5^-\).
		Since \(a\) and \(b\) have odd degrees in~\(G'\), \(G'\) is not isomorphic to \(K_3\).
		First, suppose that \(G'\) is a Gallai graph, and let \(\D'\) be a minimum path decomposition of \(G'\).
		Let \(P_a',P_b'\in\D'\) be paths containing~\(a\) and~\(b\) as end-vertices, respectively.
		Put \(P_a = P_a' + av\), \(P_b = P_b' + bu\), and \(R = bvcua\)
		(if \(P_a' = P_b'\), we put \(P_a=P_b=P_a'+av+bu\)) (see Figure~\ref{fig:m13-a}).
		Clearly, \(\D=\D'-P_a'-P_b'+P_a +P_b + R\) is a path decomposition of \(G\)
		such that \(|\D| \leq \lfloor |V(G')|/2\rfloor + 1 \lfloor n/2\rfloor\).

                Therefore, we may assume that \(G'\simeq K_5^-\).
		and let \(V(g') = \{a,b,c,x,y\}\).
		Let \(R_1'=axby, R_2' = acy, R_3 = ayxcb\),
		and note that \(\{R_1',R_2',R_3'\}\) is a path decomposition of \(G'\).
		Note that \(R'_i\) does not contain \(u\) or \(v\), for \(i=1,2,3\),
		\(R'_1\) does not contain \(c\),
		and \(R'_2\) does not contain \(b\).
		Put \(R_1 = R_1'+ au + uc + cv\) and \(R_2 = R_2' + av + vb + bu\).
		Note that \(\{R_1,R_2,R_3\}\) is a decomposition of \(G\) (see Figure~\ref{fig:m14-a-new}).
		Therefore, \(\pn(G) = 3 =\lfloor 7/2\rfloor = \lfloor n/2\rfloor\).
                This finishes the proof of the claim.
                \begin{figure}[ht]
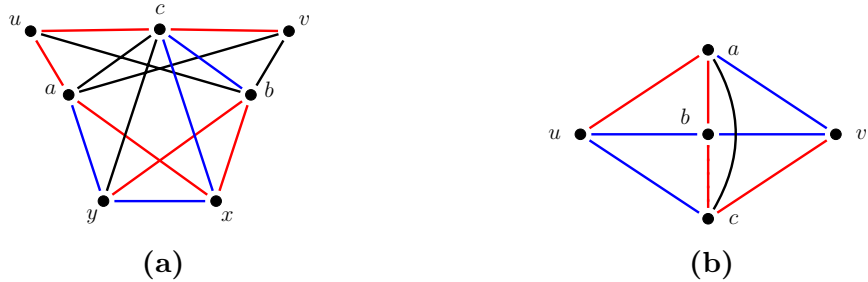

                  \centering
                  \begin{subfigure}[b]{.45\linewidth}
                    \centering\scalebox{.7}{\input{Figures/m14-a-new.tikz}}%
                    \caption{}\label{fig:m14-a-new}
                  \end{subfigure}%
                  \begin{subfigure}[b]{.45\linewidth}
		  \centering\scalebox{.7}{\input{Figures/y1.tikz}}%
		  \caption{}\label{fig:y1}
                  \end{subfigure}%
                  \caption{in Figure (a), the paths \(R_1\), \(R_2\), and \(R_3\) are illustrated, respectively,
                  		in red, black, and blue;
                  		in Figure (b), the path \(P\) and the lifted edges are illustrated, respectively,
                  		in red and blue.}
                  		\label{fig:m14}
                \end{figure}
	\end{proof}

	Now, let \(P = uabcv\) (see Figure~\ref{fig:y1}) and \(G' = (G - E(P)) +\lift{buc} + \lift{avb}\).
	By Fact~\ref{fact:remove-terminal},~\(G'\) is a partial \(3\)-tree.
	Moreover, \(G'\) is connected and, by minimality of \(G\), \(G\) is a Gallai graph or isomorphic to \(K_3\) or to \(K_5^-\).
        Since \(n \geq 6\), we have \(|V(G')| \geq 4\) and therefore \(G \not\simeq K_3\).
	If \(G'\simeq K_5^-\), then \(n = 7\) and \(G' -\lift{buc} -\lift{avb} = G-E(P)\) is a proper subdivision of \(K_5^-\).
	By Proposition~\ref{prop:K5--simple-subdivision}, \(\pn(G-E(P)) = 2\).
	Hence, \(\pn(G) = 3 = \floor{n/2}\).
	If~\(G'\) is a Gallai graph, then \(\pn(G')\leq \floor{|V(G')|/2}\) and \(\pn(G) \leq \floor{|V(G')|/2} +1 = \floor{n/2}\).
	This concludes the proof of Property~\ref{step:2}.
\end{proof}

\begin{property}\label{step:3}
If \(a\) is a common neighbor of two terminals, then \(d(a)\) is odd.
\end{property}

\begin{proof}
\noindent
From the Properties \(1\) and \(2\), if \(u,v\) are two terminals of \(G\), we have \(N(u)\cap N(v) =\{a,b\}\).
Here, we prove that \(d(a),d(b)\) are odd.

Let \(u\) and \(v\) be two terminals of \(G\),
let \(N(u) = \{a,b,c\}, N(v) = \{a,b,d\}\),
and \(O = ua+ub+uc+va+vb+vd\).
Suppose for a contradiction that \(d(a)\) is even.
Now, suppose that \(ac\notin E(G)\).
Let \(G' = G - u - v + ac\).
It is clear that \(d_{G'}(a) = d_G(a) - 1\), hence \(d_{G'}(a)\) is odd.
Let \(C'\) be the component of \(G'\) containing the edge \(ac\).
Note that \(C'\not\simeq K_3\) because \(d_{C'}(a) = d_{G'}(a)\) is odd.
Now, put \(C = C' - \lift{auc}\), and note that \(\pn(C)\leq\lfloor|V(C')|/2\rfloor\),
(if \(C'\simeq K_5^-\), then we use Proposition~\ref{prop:K5--simple-subdivision}).
Note that \(d_C(a) = d_{C'}(a)\),
thus, if \(\D_C\) is a minimum path decomposition of \(C\),
there is a path \(P_a\) in \(\D_C\) that has \(a\) as an end-vertex.
Note also that \(C\) does not contain \(v\).
Therefore, \(P = P_a + av\) is a path.
Put \(Q = ubvd\) (see Figure~\ref{fig:y2}) and note that \(\D_C -P_a + P + Q\) is a path decomposition of \(H = C'-ac + O\).
Moreover, the vertices \(u\) and \(v\), and the vertices in \(V(C')\) are isolated in \(G-E(H)\),
and \(\pn(H) \leq \pn(C) + 1 \leq\lfloor|V(C')|/2\rfloor + 1\).
Therefore, \(H\) is a \(\big(\lfloor|V(C')|/2\rfloor + 1\big)\)-reducing subgraph of \(G\),
a contradiction to Claim~\ref{claim:no-reducing}.
Therefore, we can suppose that \(ac\in E(G)\).
By symmetry, we have \(ad\in E(G)\).

\begin{figure}[h]
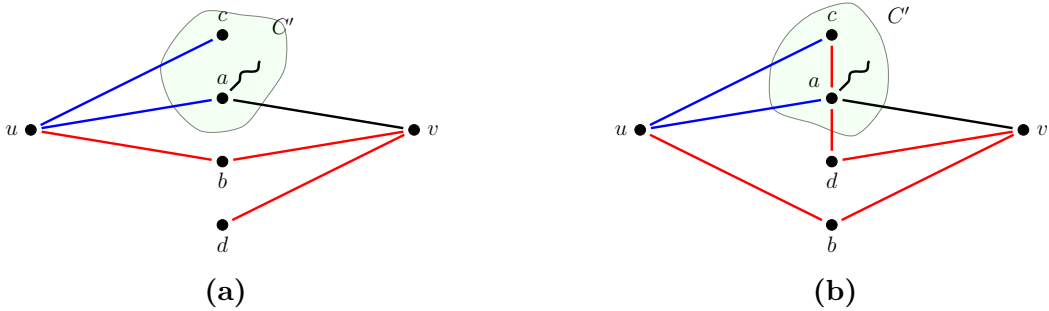

  \centering
  \begin{subfigure}[b]{.5\linewidth}
    \centering\scalebox{.7}{\input{Figures/y2.tikz}}%
    \caption{}\label{fig:y2}
  \end{subfigure}%
  \begin{subfigure}[b]{.5\linewidth}
    \centering\scalebox{.7}{\input{Figures/y3.tikz}}%
    \caption{}\label{fig:y3}
  \end{subfigure}%
  \caption{the paths \(P\) and \(Q\) are illustrated, respectively, in black and red.}
\end{figure}

Now, let \(O = ua+ub+uc+va+vb+vd+ac+ad\), and let \(G' = G - u - v - ad\).
It is clear that \(d_{G'}(a) = d_G(a) - 3\), hence \(d_{G'}(a)\) is odd.
Let \(C'\) be the component of \(G'\) containing the edge \(ac\).
In what follows, the proof is similar to the case above.
We have \(C'\not\simeq K_3\) because \(d_{C'}(a) = d_{G'}(a)\) is odd.
Let \(C = C' -\lift{auc}\) and note that \(d_C(a)\) is also odd,
and \(\pn(C)\leq\lfloor|V(C')|/2\rfloor\).
Let \(\D_C\) be a minimum path decomposition of \(C\),
and let \(P_a\in\D_C\) be a path that has \(a\) as an end-vertex.
Again, \(P_a\) does not contain \(v\) because \(v\notin V(C)\),
hence \(P = P_a + av\) is a path.
Put \(Q = ubvdac\) (see Figure~\ref{fig:y3}) and note that \(\D_C -P_a + P + Q\)
is a path decomposition of \(H = (C'-ac) + O\).
Analogously to the case above, \(H\) is a \(\big(\lfloor|V(C')|/2\rfloor + 1\big)\)-reducing subgraph of \(G\),
a contradiction to Claim~\ref{claim:no-reducing}.
Therefore, we conclude that \(d(a)\) is odd.
\end{proof}

\noindent
Now, we can conclude our proof.
Let \(u_1,\ldots,u_k\) be the terminals of \(G\).
Since \(n \geq 6\), we have~\(k \geq 2\).
By Claim~\ref{claim:every-terminal-degree-3}, Property~\ref{step:1}, and Property~\ref{step:2},
we have \(|N(u_i) \cap N(u_j)| = 2\) for every~\(i \neq j\).
In what follows, the proof is divided into two cases,
depending on whether there are two vertices, say \(a,b\), such that \(a,b\in N(u_i)\) for every \(i=1,\ldots,k\).
First, suppose that such pair of vertices does not exist.
Let \(N(u_1) = \{x_1,x_2,x_3\}\).
Since \(|N(u_1)\cap N(u_2)| = 2\), we may suppose, without loss of generality, that \(N(u_2) = \{x_1,x_2,x_4\}\),
where \(x_4\neq x_3\).
Since there are no two vertices contained in \(N(u_i)\), for every \(i=1,\ldots,k\),
there is a terminal, say \(u_3\), such that \(x_1\) or \(x_2\) is not contained in \(N(u_3)\).
Suppose, without loss of generality, that \(x_2\notin N(u_3)\).
Thus, \(N(u_3) = \{x_1,x_3,x_4\}\),
because \(|N(u_3)\cap N(u_1)|=|N(u_3)\cap N(u_2)|=2\).
 Now, suppose that there is a vertex \(x_5\) different from \(x_1,\ldots,x_4\)
 contained in the neighbor of a terminal vertex, say \(u\).
 It is not hard to check that \(|N(u)\cap N(u_i)| < 2\) for at least one \(i\in\{1,2,3\}\),
 a contradiction to Property~\ref{step:1}.

 Now we use an underlying \(3\)-tree \(G^*\) of \(G\).
 Since \(u_1,\ldots,u_k\) are terminals, the set~\(N(u_i)\) induces a complete graph in \(G^*\),
 i.e., \(x_ix_j\in E(G^*)\) for every \(i,j\in\{1,2,3,4\}\), with \(i\neq j\).
 By Fact~\ref{fact:remove-terminal}, \(G^+= G^* - u_1-\cdots - u_k\) is a \(3\)-tree.
 If \(G^+\) has a vertex different from \(x_1,x_2,x_3,x_4\),
 then \(G^+\) contains at least \(5\) vertices, and by Facts~\ref{fact:terminal-vertices} and~\ref{fact:at-least-two-terminals},
 the graph \(G^+\) contains two non-adjacent terminals.
 Thus, \(G^+\) contains a terminal \(u^+\) which is not in \(\{x_1,x_2,x_3,x_4\}\),
 but then \(u^+\) is a terminal of \(G^*\),
 and, consequently, \(u^+\) is a terminal of~\(G\) different from \(u_1,\ldots, u_k\), a contradiction.
 Therefore,~\(G\) consists only of the vertices \(u_1,\ldots,u_k,x_1,x_2,x_3,x_4\).
 By Claim~\ref{claim:every-terminal-degree-3}, \(u_1,\ldots,u_k\) have degree \(3\).
 Note that \(x_1,x_2\in N(u_1)\cap N(u_2)\),
 \(x_3\in N(u_1)\cap N(u_3)\), and \(x_4\in N(u_2)\cap N(u_3)\).
 Thus,~\(x_j\) is a common neighbor of at least two terminals of \(G\), for \(j=1,\ldots,4\).
 By Property~\ref{step:3}, \(x_1,x_2,x_3,x_4\) have odd degree.
 Therefore,~\(G\) is an odd Graph, and by Theorem~\ref{thm:lovasz}, we have \(\pn(G) = n/2\).
 We note that, by using Property~\ref{step:2}, one can prove that \(G\) is a graph obtained from \(K_{4,4}\) by removing a perfect matching.

Now, suppose that there are two vertices, say \(a,b\), such that \(a,b\in N(u_i)\), for every \(i=1,\ldots,k\).
Note that this is also true for the underlying spanning \(3\)-tree \(G^*\) of \(G\),
i.e.,~\(G^*\) is a double centered \(3\)-tree with center in \(a,b\).
Thus, by Proposition~\ref{prop:double-centered-3-tree},
there is a tree \(T\) such that \(G^* =  T\vee K_2\), where \(V(G^*) = V(T) +a+b\).
Now, if \(C\) is a cycle in \(G^*\), then \(C\) must contain either \(a\) or \(b\),
otherwise \(V(C)\subseteq V(T)\), but \(T\) has no cycles.
It is clear that if \(C\) is cycle in \(G\), then \(C\) is a cycle in \(G^*\),
hence \(C\) contains \(a\) or \(b\).
Thus, every cycle of~\(G\) must contain \(a\) or \(b\),
which have odd degree by Property~\ref{step:3}.
Therefore, by Theorem~\ref{thm:pyber}, we have \(\pn(G) = \lfloor n/2\rfloor\).
This concludes the proof.
\end{proof}

The next corollary comes from the fact that every graph with no subdivision of \(K_4\) is a partial \(3\)-tree.
In fact, it is not hard to check that graphs with no subdivision of \(K_4\) do not contain
any of the forbidden minors for partial \(3\)-trees.

\begin{corollary}\label{cor:K4-free}
	Let \(G\) be a connected graph with \(n\) vertices and with no subdivision of~\(K_4\).
	Then \(\pn(G)\leq\lfloor n/2 \rfloor\)
	or \(G\) is isomorphic to \(K_3\).
\end{corollary}

\subsection{Haj\'os' Conjecture for graphs with treewidth at most~3}\label{sec:hajos}

Before verifying Conjecture~\ref{conj:hajos} for graphs with treewidth at most \(3\),
we give an analogous definition of reducing subgraph for dealing with Conjecture~\ref{conj:hajos}.
Let \(G\) be an Eulerian graph, and let \(H\) be an Eulerian subgraph of \(G\).
Given a positive integer \(r\), we say that~\(H\) is an \emph{\(r\)-cycle reducing subgraph} of \(G\)
if \(G-E(H)\) has at least \(2r\) isolated vertices and \(\cn(H)\leq r\).
If \(H\) is an \(1\)-cycle reducing subgraph, we say that \(H\) is a \emph{reducing cycle} of~\(G\),
and we say that \(H\) is a \emph{cycle reducing subgraph} if \(H\)
is an \(r\)-cycle reducing subgraph for some positive integer \(r\).
We say a graph \(G\) with \(n\) non-isolated vertices
is a \emph{Haj\'os graph} if  \(\cn(G) \leq \lfloor (n-1)/2\rfloor\).
The following lemma holds, and its proof is analogous to the proof of Lemma~\ref{lemma:reducing}.

\begin{lemma}\label{lemma:cycle-reducing}
	Let \(G\) be an Eulerian graph and \(H\subset G\) be a cycle reducing subgraph of \(G\) such that \(H\neq G\).
	If \(G-E(H)\) is a Haj\'os graph, then \(G\) is a Haj\'os graph.
\end{lemma}

Note that the statement of Lemma~\ref{lemma:cycle-reducing} does not require \(G\) to be connected.
In fact, it is easy to check that if a graph \(G\) is a vertex-disjoint union of two Haj\'os graphs, then \(G\) is also a Haj\'os graph.

\begin{theorem}\label{thm:hajos-treewidth3}

	Let \(G\) be a connected Eulerian partial \(3\)-trees with \(n\) vertices.
	Then \(\cn(G)\leq\lfloor (n-1)/2\rfloor\).
\end{theorem}

\begin{proof}

	Suppose, for a contradiction, that the statement does not hold,
	and let \(G\) be a counter-example that minimizes \(n\).
	The statement holds trivially if~\(n=3\) or~\(n=4\).
        Thus, we may assume that \(n \geq 5\).
        We claim that \(G\) is \(2\)-connected.
        Indeed, suppose that \(G\) contains a cut-vertex \(v\),
        and let \(B_1,\ldots,B_k\) be the \(2\)-connected components of \(G\).
        It is easy to see that \(|V(G)| = \big(\sum_{i=1}^k|V(B_i)|\big) -k +1\),
        i.e., \(|V(G)| -1= \big(\sum_{i=1}^k|V(B_i)|\big) -k\).
        Since \(G\) is minimal, \(\cn(B_i)\leq\lfloor(|V(B_i)|-1)/2\rfloor\), for \(i=1,\ldots,k\).
        Let \(\D_i\) be a minimum cycle decomposition of \(B_i\).
        Thus, \(\D=\cup_{i=1}^k \D_i\) is a cycle decomposition of \(G\)
        of size at most
        \[\sum_{i=1}^k \left\lfloor \frac{|V(B_i)|-1}{2}\right\rfloor
        \leq \left\lfloor\sum_{i=1}^k  \frac{|V(B_i)|-1}{2}\right\rfloor
        =\left\lfloor\frac{1}{2}\left(\left(\sum_{i=1}^k  |V(B_i)|\right)-k\right)\right\rfloor
        =\left\lfloor\frac{|V(G)|-1}{2}\right\rfloor.\]
        Therefore, we may assume that \(G\) is \(2\)-connected.
	Since \(n\geq 5\), \(G\) contains at least two terminals, say~\(u\) and~\(v\), and
	since \(G\) is Eulerian, \(d(u) = d(v) = 2\).
	Since \(G\) is \(2\)-connected, there is a cycle \(C\) in \(G\) containing \(u\) and \(v\).
	The vertices \(u\) and \(v\) are isolated in \(G' = G-E(C)\), hence either \(G=C\) or \(C\) is a reducing cycle.
	Since \(G\) is a minimal counter example, \(G'\) is a Haj\'os graph,
	hence by Lemma~\ref{lemma:cycle-reducing}, \(G\) is a Haj\'os graph.
\end{proof}

\begin{corollary}\label{cor:K4-free-hajos}

	Let \(G\) be a connected Eulerian graph with \(n\) vertices and with no subdivision of~\(K_4\).
	Then \(\cn(G)\leq\lfloor n-1/2 \rfloor\).
\end{corollary}

\section{Graphs with maximum degree 4}\label{sec:max4}

In this section, we verify Conjectures~\ref{conj:gallai} and~\ref{conj:hajos}
for graphs with maximum degree at most \(4\).
Although this case was already verified~\cite{BonamyPerrett16+,GranvilleMoisiadis87} for both conjectures,
the proofs presented here exemplifies the application of reducing subgraphs.

\subsection{Gallai's Conjecture for graphs with maximum degree~$4$}\label{sec:max4-gallai}

\begin{theorem}\label{theorem:max4}
	Let \(G\) be a connected graph with \(n\) vertices.
	If \(G\) has maximum degree \(4\),
	then \(\pn(G)\leq \lfloor n/2\rfloor\)
	or \(G\) is isomorphic to \(K_3\), \(K_5\) or to \(K_5^-\).
\end{theorem}

\begin{proof}
	Suppose, for a contradiction, that the statement does not hold,
	and let \(G\) be a counter-example that minimizes \(n\).
	It is easy to check that the statement is true for graphs with at most five vertices.
	Thus, we may suppose \(n\geq 6\).
	The following claims follow analogously to the the proof of Theorem~\ref{theorem:treewidth-most-three},
	therefore we omit their proofs.

	\begin{claim}\label{claim:no-reducing-max4}
		\(G\) contains no reducing subgraph.
	\end{claim}

	\begin{claim}\label{cl:twm3:not-contain-useful-bridge-max4}
	$G$ contains no useful cut-edge.
	\end{claim}

	By Corollary~\ref{cor:K4-free}, we may suppose that \(G\) contains a subdivision of \(K_4\).
	Thus, let \(H\) be a subdivision of \(K_4\) in \(G\) with a minimum number of edges,
	and let \(x_1,x_2,x_3,x_4\) be the vertices of degree \(3\) in \(H\).
	From now on, for each \(i,j\in\{1,2,3,4\}\) with \(i\neq j\),
	let~\(P_{i,j}\) be the path in~\(H\) joining~\(x_i\) to~\(x_j\),
	where for \(\{i,j\}\neq\{i',j'\}\), the paths~\(P_{i,j}\) and~\(P_{i',j'}\) have no internal vertex in common.
	Let \(S\) be the set of edges in \(G-E(H)\) incident to \(x_1,x_2,x_3,x_4\).
	Since \(G\) has maximum degree~\(4\),
	there is at most one edge in \(S\) incident to \(x_i\), for \(i=1,2,3,4\).
	For each \(i\in\{1,2,3,4\}\), if there is an edge~\(e_i\) in~\(S\)
	incident to~\(x_i\), then let \(e_i = x_iz_i\),
	and put \(Z = \{z_i\colon x_iz_i\in S\}\).

\begin{claim}
	\(|Z\cap V(H)|\leq 1\),
	and if \(Z\cap V(H)= \{z_i\}\), then \(z_i\neq z_j\) for every \(j\in\{1,2,3,4\}\) with \(i\neq j\).
\end{claim}

\begin{proof}
	Suppose that \(z_1\in V(H)\).
	If \(z_1\in V(P_{1,j})\) for any \(j\in\{2,3,4\}\),
	then \(H+x_1z_1 - P_{1,j}(x_1,z_1)\) is a subdivision of \(K_4\) with less edges than \(H\).
	Thus, we can suppose, without loss of generality, that \(z_1 \in V(P_{2,3})\).
	If \(P_{1,j}\) has length at least \(2\), then \(H+x_1z_1-P_{1,j}\) is a subdivision of \(K_4\) with less edges than \(H\).
	Therefore, \(P_{1,j} = x_1x_j\), for \(i=2,3,4\).
	Analogously, if \(z_i\in V(H)\), then \(P_{i,j} = x_ix_j\) for \(j\in \{1,2,3,4\}- i\).
	Since \(z_1\in P_{2,3}\), the length of \(P_{2,3}\) is at least \(2\),
	hence \(z_2,z_3\notin V(H)\).
	Therefore \(|Z\cap V(H)|\leq 2\).
	Now, suppose \(z_4\in V(H)\).
	We have \(z_4\notin V(P_{4,j})\) for \(j=1,2,3\).
	Since~\(P_{1,j}\) has length \(1\) for \(j=2,3,4\),
	we have \(z_4\notin V(P_{1,j})\).
	Therefore, \(z_4\in V(P_{2,3})\),
	hence \(H+x_1z_1+x_4z_4 - x_2\) contains a subdivision of \(K_4\) with less edges than \(H\).
	Therefore \(z_4 \notin V(H)\) and \mbox{\(|Z\cap V(H)|\leq 1\)}.
\end{proof}

	In what follows we divide the proof on whether \(|Z\cap V(H)| = 0\) or \(|Z\cap V(H)| =1\).

	\smallskip
	\noindent\textbf{(i) \(|Z\cap V(H)| = 0\).}
	First, suppose that \(d_S(z_i) = 4\) for some \(i\).
	Thus, we have \(z_1=z_2=z_3=z_4\).
	We claim  that \(P_{i,j}\) has length at most \(1\), for every \(i,j\in\{1,2,3,4\}\) with~\(i\neq j\).
	Indeed, suppose, without loss of generality, that \(P_{1,2}\) has length at least \(2\).
	We have \(\sum_{j=2}^4 |E(P_{1,j})|\geq 4\).
	Thus, \(H+S-x_1\) contains a subdivision of \(K_4\) in \(G\)
	with at most \(|E(H)| -1\) edges,
	a contradiction to the minimality of~\(H\).
	Therefore \(G\) is isomorphic to~\(K_5\), and the result follows.
	Now, suppose \(d_S(z_i) \leq 2\)
	and, without loss of generality, suppose that \(z_1\neq z_4\) and \(z_2\neq z_3\).
	Let \(Q'_{1,4} = z_1x_1 + P_{1,2} + P_{2,3} + P_{3,4} + x_4z_4\)
	and \(Q'_{2,3} = z_2x_2+P_{2,4} + P_{4,1} + P_{1,3} +x_3z_3\)
	(we may ignore the edges \(x_iz_i\) if \(x_iz_i\notin S\), see Figure~\ref{fig:i-2}).
	Clearly, \(\D_{H+S}=\{Q'_{1,4},Q'_{2,3}\}\) is a path decomposition of \(H+S\),
	and the vertices \(x_1,x_2,x_3,x_4\) are isolated in \(G' = G - E(H) - S\).
	Thus, \(H+S\) is a \(2\)-reducing subgraph of \(G\),
	a contradiction to Claim~\ref{claim:no-reducing-max4}.

        \begin{figure}[ht]
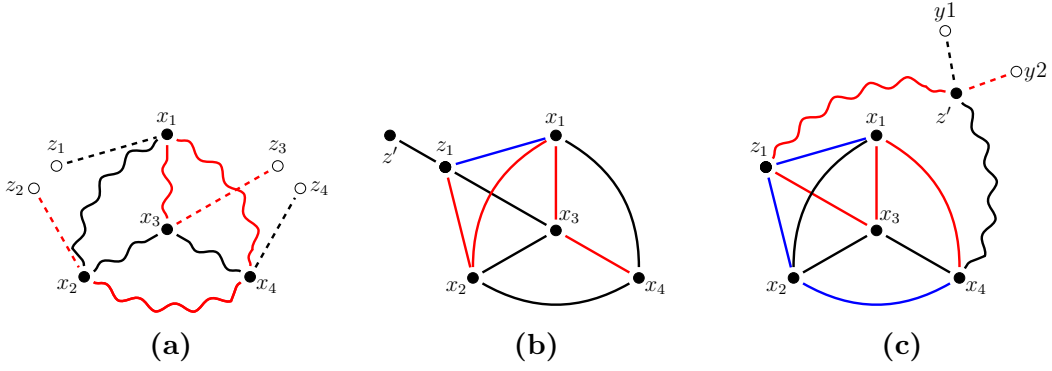

          \centering
          \begin{subfigure}[b]{.3\linewidth}
            \centering\scalebox{.7}{\input{Figures/i-2.tikz}}%
            \caption{}\label{fig:i-2}
          \end{subfigure}%
          \begin{subfigure}[b]{.3\linewidth}
            \centering\scalebox{.7}{\input{Figures/z-1.tikz}}%
            \caption{}\label{fig:z-1}
          \end{subfigure}%
          \begin{subfigure}[b]{.3\linewidth}
            \centering\scalebox{.7}{\input{Figures/z-2.tikz}}%
            \caption{}\label{fig:z-2}
          \end{subfigure}%
          \caption{illustrations of the case $|Z\cap V(H)|=0$.}
        \end{figure}

	Thus, assume that \(d_S(z_i) = 3\) for some \(i\).
	We may suppose, without loss of generality, that \(z_1=z_2=z_3\).
	Note that \(H'=H+x_1z_1+x_2z_2+x_3z_3\) is a subdivision of~\(K_5^-\),
	where~\(z_1\) and~\(x_4\) are its vertices of degree \(3\).
	By Corollary~\ref{cor:K5--subdivision}, If \(H'\) is a proper subdivision of \(K_5^-\),
	then \(\pn(H') = 2\).
        Thus, there exist two paths \(P'\) and \(Q'\) such that \(H' = P' + Q'\).
	Since \(x_4\) has degree \(3\) in \(H'\),
	one between \(P'\) and \(Q'\), say \(P'\), has \(x_4\) as an end-vertex.
	If \(x_4z_4\in S\), then put \(P = P'+x_4z_4\).
	Thus, \(\{P,Q'\}\) is a decomposition of \(H+S\) into two paths,
	and the vertices \(x_1,x_2,x_3,x_4\) are isolated vertices of \(G' = G - E(H) - S\).
	Thus, \(H+S\) is a \(2\)-reducing subgraph of \(G\),
	a contradiction to Claim~\ref{claim:no-reducing-max4}.

	Now, suppose \(H'\) is isomorphic to \(K_5^-\).
	If \(G=H'\), then the statement holds.
	Thus, we can suppose \(G' = G-E(H')\) is not empty.
	Moreover, \(d_{G'}(z_1)+d_{G'}(x_4)\geq 1\).
	Suppose, without loss of generality, that \(d_{G'}(z_1)=1\),
	and let \(z'\) be the neighbor of \(z_1\) in \(G'\).
	Suppose that there is no path in \(G'\) joining \(z_1\) and \(x_4\),
	then \(z_1z'\) is a cut-edge, and by Claim~\ref{cl:twm3:not-contain-useful-bridge-max4},
	we have \(d(z') = 1\).
	If \(G = H' + z_1z'\), then \(\{z'z_1x_3x_2x_4x_1,z_1x_2x_1x_3x_4,x_1z_1\}\) (see Figure~\ref{fig:z-1})
	is a decomposition of \(G\) with size \(3 = n/2\).
	Thus, we may suppose that the vertex \(x_4\) has a neighbor \(x'\) in \(G'\).
	Analogously to the case above \(d(x')=1\) and,
        therefore, any path joining \(x'\) to \(z'\) is a reducing path,
	a contradiction to Claim~\ref{claim:no-reducing-max4}.

	Thus, we may suppose that there is a path in \(G'\) joining \(z_1\) to \(x_4\)
	and let \(P'\) be the shortest path in \(G'\) with this property.
	If \(P'\) has length \(1\), then \(G\) is isomorphic to \(K_5\), and the result follows.
	Thus, let \(z'\) be an internal vertex of \(P'\)
	and let \(S'\) be edges incident to \(z'\) in \(G-E(H')-E(P')\).
	Put \(H''=H'+P'+S'\).
	Note that the edges in \(S'\) are not incident to vertices of \(H'\),
	because \(d_{H'+P'}(v) = 4\) for every \(v\in V(H')\).
	Also, by the minimality of~\(P'\), the edges in \(S'\) are not incident to other vertices
	of \(P'\).
	Suppose \(S' = \{z'y_1,z'y_2\}\)
	and note that \(\{y_1z'+P'(z',x_4) + x_4x_3x_2x_1,y_2z'+P'(z',z_1)
						+ z_1x_3x_1x_4,x_1z_1x_2x_4\}\)
	is a path decomposition of \(H''\)
	(we may ignore the edges \(z'y_1,z'y_2\) if \(z'y_1,z'y_2\notin E(G)\), see Figure~\ref{fig:z-2}).
	Now, note that \(z'\) and every vertex of \(H'\) is isolated in \(G-E(H'')\).
	Thus \(G-E(H'')\) has at least \(6\) isolated vertices.
	Therefore, \(H''\) is a \(3\)-reducing subgraph of \(G\),
	a contradiction to Claim~\ref{claim:no-reducing-max4}.
        This finishes the proof of case (i).

	\smallskip
	\noindent\textbf{(ii) \(|Z\cap V(H)| = 1\).}
	Put \(G' = G - E(H) - S\).
	Suppose, without loss of generality, that \(z_1\in Z\cap V(P_{2,3})\).
	As seen above, \(z_1 \neq z_2,z_3,z_4\).
	First, suppose \(d_S(z_i) \leq 2\).
	By symmetry, we can suppose \(z_2\neq z_4\).
	Let \(Q'_{2,4} = z_2x_2 + P_{2,3}  + x_3x_1x_4z_4\) and  \(Q'_{1,3} = z_1x_1x_2 + P_{2,4} + P_{4,3} + x_3z_3\)
	(we may ignore the edges \(x_iz_i\) if \(x_iz_i\notin S\), see Figure~\ref{fig:z-3}).
	Clearly, \(\D_{H+S}=\{Q'_{2,4},Q'_{1,3}\}\) is a path decomposition of \(H+S\)
	and the vertices \(x_1,x_2,x_3,x_4\) are isolated vertices of \(G'\).
	Thus, \(H+S\) is a \(2\)-reducing subgraph of \(G\),
	a contradiction to Claim~\ref{claim:no-reducing-max4}.

	Now, suppose \(d_S(z_i) = 3\) for some \(i\in\{2,3,4\}\), i.e., \(z_2=z_3=z_4\).
	If \(P_{i,4}\) has length at least \(2\) for any \(i\in\{1,2,3\}\),
	then \(H+S-x_4\) contains a subdivision of \(K_4\) with less than \(|E(H)|\) edges.
	Consider the paths \(Q_1' = z_1x_1x_2x_4x_3z_4\) and \(Q_2' = z_4x_2 + P_{2,3} + x_3x_1x_4\) (see Figure~\ref{fig:z-4}).
	If \(z_1z_4\in E(G)\), then \(\{Q_1',Q_2',z_1z_4x_4\}\) decompose \(H'=H +S+z_1z_4\)
	and \(z_1,z_4,x_1,x_2,x_3,x_4\) are isolated in \(G - H'\),
	hence \(H'\) is a \(3\)-reducing subgraph of \(G\),
	a contradiction to Claim~\ref{claim:no-reducing-max4}.
	Thus, we may suppose that \(z_1z_4\notin E(G)\).
	If there is an edge \(z_1z'\) in \(G'\) incident to \(z_1\), then put \(Q_1 = Q_1'+z_1z'\), otherwise, put \(Q_1 = Q_1'\).
	Now, note that \(x_1,x_2,x_3,z_1\) are isolated vertices in \(G - E(Q_1) - E(Q_2')\),
	hence \(Q_1+Q_2'\) is a \(2\)-reducing subgraph of \(G\),
	a contradiction to Claim~\ref{claim:no-reducing-max4}.
\end{proof}

\begin{figure}[h]
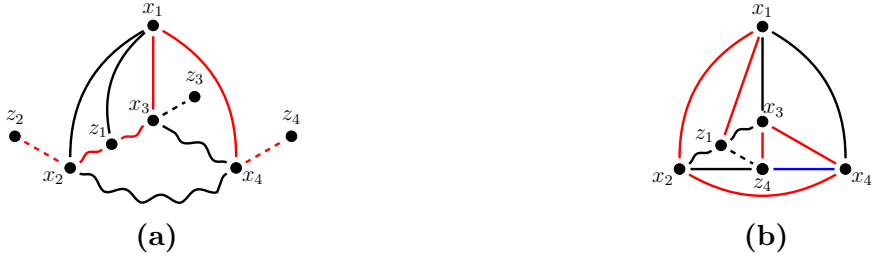

	\centering
	\begin{subfigure}[b]{.5\linewidth}
	\centering\scalebox{.7}{\input{Figures/z-3.tikz}}
	\caption{}\label{fig:z-3}
	\end{subfigure}%
	\begin{subfigure}[b]{.5\linewidth}
	\centering\scalebox{.7}{\input{Figures/z-4.tikz}}
	\caption{}\label{fig:z-4}
	\end{subfigure}%
	\caption{illustrations of the case $|Z\cap V(H)|=1$.}
\end{figure}

\subsection{Haj\'os Conjecture for graphs with maximum degree~$4$}\label{sec:max4-gallai}

In this section we present a new proof for the result given by Granville and Moisiadis~\cite{GranvilleMoisiadis87},
which verifies Conjecture~\ref{conj:hajos} for graphs with maximum degree \(4\).
For the base cases in our proof we use the following result of Jackson~\cite{jackson80}.

\begin{theorem}[Jackson, 1980]\label{thm:jackson}
	Every \(2\)-connected \(k\)-regular graph with at most \(3k\) vertices is Hamiltonian.
\end{theorem}

The following lemma is useful to obtain cycle decompositions from closed trails in a graph with maximum degree \(4\).

\begin{lemma}\label{lemma:double-paths}
	Let \(G\) be a graph composed of two paths \(P_1\) and \(P_2\)
	that share the same end-vertices \(x\) and \(y\),
	and let \(S = V(P_1)\cap V(P_2) - x - y\).
	Then \(\cn(G)\leq |S|+1\).
\end{lemma}

\begin{proof}
	Let \(G\), \(P_1\), \(P_2\), \(x\), \(y\) and \(S\) be as in the statement.
	The proof follows by induction on~\(s=|S|\).
	If \(s=0\), then \(G\) is a cycle, and the statement holds.
	Thus, suppose \(s>0\) and suppose that the statement holds for \(s'<s\).
	Let \(y'\) be the vertex of \(V(P_1) \cap V(P_2)-x-y\)
	that minimizes the length of \(P_1(y',y)\).
	Put \(P_1' = P_1(x,y')\) and \(P_2'=P_2(x,y')\).
	Clearly \(|V(P_1')\cap V(P_2')-x-y'| =s'\leq s-1\).
	By the induction hypothesis, \(P_1' \cup P_2'\) admits a decomposition \(\D'\)
	into at most \(s'+1 \leq s\) cycles.
  Note that \(P_1(y',y)\) contains no vertex of~\(P_2\),
	thus, \(C=P_1(y',y)\cup P_2(y',y)\) is a cycle,
	and \(\D+C\) is a decomposition of \(G\) of size at most~\(s+1\).
\end{proof}

\begin{theorem}
	If \(G\) is an Eulerian graph with \(n\) non-isolated vertices and maximum degree \(4\),
	then \(\cn(G)\leq\lfloor (n-1)/2\rfloor\).
\end{theorem}

\begin{proof}
	Suppose, for a contradiction, that the statement does not hold,
	and let \(G\) be a counter-example that minimizes \(n\).
	Analogously to the proof of Theorem~\ref{thm:hajos-treewidth3},
	we can prove that \(G\) is \(2\)-connected.
	Moreover, \(G\) contains at most one vertex of degree \(2\),
	otherwise \(G\) contains a reducing cycle.
	Thus, \(G\) contains at least one vertex of degree~\(4\),
	hence \(n\geq 5\).
	Now, suppose that \(n\leq 8\).
	We claim that \(G\) contains a Hamiltonian cycle.
	If \(G\) contains a vertex \(v\) of degree \(2\),
	then let \(x,y\) be the neighbors of \(v\),
	and we replace~\(v\) by a copy of~\(K_5^-\)
	in which the each vertex of degree~\(3\)
	is joined to one between~\(x\) and~\(y\).
	The graph~\(G'\) obtained is \(4\)-regular and contains \(n+4\) vertices.
	If \(G\) does not contain a vertex of degree \(2\), then we put \(G' = G\).
	Since \(n\leq 8\), \(G'\) has at most \(12\) vertices,
	and by Theorem~\ref{thm:jackson}, \(G'\) contains a Hamiltonian cycle \(C'\).
	It is not hard do check that we can obtain from \(C'\) a Hamiltonian cycle \(C\) of \(G\).
	Now, if \(n=5\), then \(G\) must be a copy of \(K_5\) and \(\cn(G) = 2\).
	Suppose that \(n=6\).
	If \(G-E(C)\) is a cycle, then \(\cn(G) = 2\).
	If \(G-E(C)\) consists of copies of \(C_3\),
	then \(G\) is the graph obtained from \(K_6\) by removing a perfect matching.
	It is not hard to check that, in this case, \(G\) can be decomposed into two Hamiltonian cycles.
	Therefore, we may suppose \(n\geq 7\), then \(G-E(C)\) is a \(2\)-regular graph and consists of two vertex-disjoint cycles,
	hence \(\cn(G) \leq 3\), and the statement holds.

	Thus, from now on, we suppose that \(n\geq 9\).
	By Corollary~\ref{cor:K4-free-hajos},
	we may suppose that \(G\) contains a subdivision of \(K_4\), say \(H\).
	Let \(x_1,x_2,x_3,x_4\) be the vertices of \(H\) with degree \(3\).
From now on, for each \(i,j\in\{1,2,3,4\}\) with \(i\neq j\),
let~\(P_{i,j}\) be the path in \(H\) joining~\(x_i\) to~\(x_j\),
where for \(\{i,j\}\neq\{i',j'\}\), the paths~\(P_{i,j}\) and~\(P_{i',j'}\)
have no internal vertex in common.
	Let~\(G'\) be the graph obtained from \(G-E(H)\) by adding a new vertex \(v\)
	adjacent to \(x_1,x_2,x_3,x_4\).
	Clearly, \(G'\) is an Eulerian graph, hence admits a cycle decomposition \(\D'\).
	Let \(C_1',C_2'\in\D'\) be cycles containing \(v\).
	Suppose, without loss of generality, that \(C_1'\) contains \(x_1,x_2\)
	and \(C_2'\) contains \(x_3,x_4\).
	Let \(Q_1 = C_1'-v\) and \(Q_2 = C_2'-v\),
	and put \(H^* = H + Q_1+Q_2\).
	The graph \(H^*\) is an Eulerian graph and \(x_1,x_2,x_3,x_4\)
	are isolated vertices in \(G^*=G-E(H^*)\).
	We claim that \(H^*\) is a cycle reducing subgraph of \(G\).
	For each \(k=1,2\) and \(i,j\in\{1,2,3,4\}\) with \(i\neq j\),
	let \(S^k_{i,j}\) be the set of internal vertices of \(P_{i,j}\) in \(V(Q_k)\),
	and put \(s^k_{i,j} = |S^k_{i,j}|\).
	Furthermore, put \(\Sigma = \sum_{k=1,2,\\ i,j\in \{1,2,3,4\}\\i\neq j} s^k_{i,j}\).
	In what follows, we give six decompositions of \(H^*\) into circuits.
	On each of these decompositions, we use Lemma~\ref{lemma:double-paths},
	to obtain a decomposition of \(H^*\) into cycles.
	Let \(\D_1 = \big\{Q_1+P_{1,4} + P_{4,3} + P_{3,2},Q_2+P_{3,1} + P_{1,2}+P_{2,4}\big\}\) (see Figure~\ref{fig:d1}).
	By Lemma~\ref{lemma:double-paths}, \(Q_1+P_{1,4} + P_{4,3} + P_{3,2}\)
	admits a decomposition into \(s^1_{1,4} + s^1_{4,3} + s^1_{3,2} + 1\) cycles,
	and \(Q_2+P_{3,1} + P_{1,2}+P_{2,4}\)
	admits a decomposition into \(s^2_{3,1} + s^2_{1,2} + s^2_{2,4} + 1\) cycles,
	hence \(H^*\) admits a decomposition into
	\(r_1=2 + s^1_{1,4} + s^1_{4,3} + s^1_{3,2} + s^2_{3,1} + s^2_{1,2} + s^2_{2,4}\) cycles.
	Analogously,
	with \(\D_2 = \big\{Q_1+P_{1,3} + P_{3,4} + P_{4,2},Q_2+P_{3,2} + P_{2,1}+P_{1,4}\big\}\) (see Figure~\ref{fig:d2}),
	we obtain a decomposition of \(H^*\)
	into \(r_2=2+ s^1_{1,3} + s^1_{3,4} + s^1_{4,2} + s^2_{3,2} + s^2_{2,1} + s^2_{1,4}\) cycles.
	Similarly,
	with \(\D_3 = \big\{Q_1+P_{1,2},Q_2+P_{3,2} + P_{2,4}, P_{1,3}+P_{3,4}+P_{4,1}\big\}\) (see Figure~\ref{fig:d3}),
	we obtain a decomposition of \(H^*\)
	into \(r_3 = 3 + s^1_{1,2} + s^2_{3,2}+s^2_{2,4}\) cycles
	(note that \(P_{1,3}+P_{3,4}+P_{4,1}\) is a cycle).
	Analogously,
	with \(\D_4 = \big\{Q_1+P_{1,2},Q_2+P_{3,1} + P_{1,4},P_{2,3}+P_{3,4}+P_{4,2}\big\}\) (see Figure~\ref{fig:d4}),
	we obtain a decomposition of \(H^*\)
	into \(r_4 = 3 + s^1_{1,2} + s^2_{3,1}+s^2_{1,4}\) cycles;
	with \(\D_5 = \big\{Q_1+P_{1,3} + P_{3,2},Q_2+P_{3,4},P_{1,4}+P_{4,2}+P_{2,1}\big\}\) (see Figure~\ref{fig:d5}),
	we obtain a decomposition of \(H^*\)
	into \(r_5 = 3 + s^1_{1,3} + s^1_{3,2}+s^2_{3,4}\) cycles;
	and
	with \(\D_6 = \big\{Q_1+P_{1,4} + P_{4,2},Q_2+P_{3,4},P_{1,3}+P_{3,2}+P_{2,1}\big\}\) (see Figure~\ref{fig:d6}),
	we obtain a decomposition of \(H^*\)
	into \(r_6 = 3 + s^1_{1,4} + s^1_{4,2}+s^2_{3,4}\) cycles.
        \begin{figure}[ht]
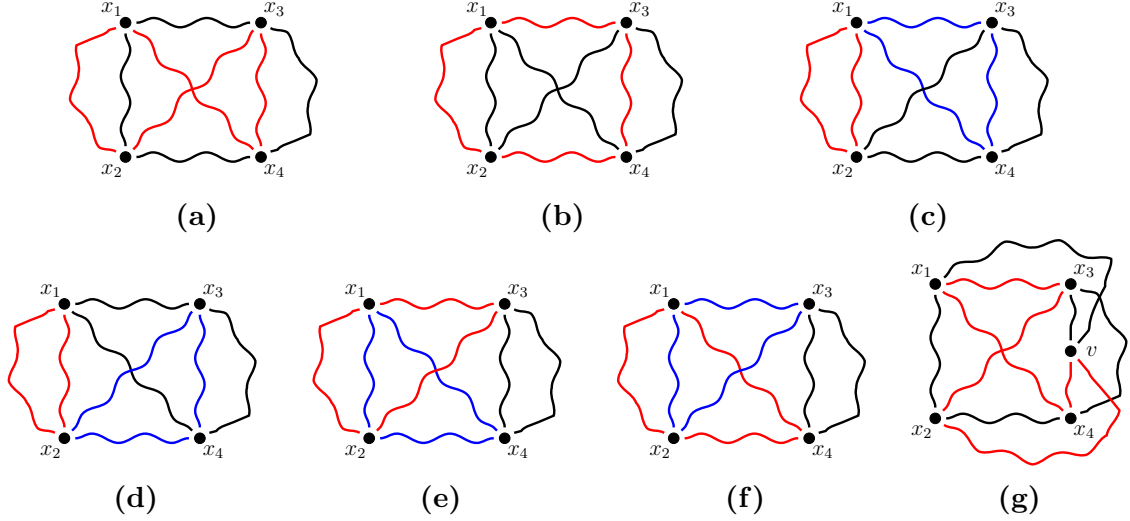

          \centering
          \begin{subfigure}[b]{.3\linewidth}
            \centering\scalebox{.7}{\input{Figures/d1.tikz}}%
            \caption{}\label{fig:d1}
          \end{subfigure}%
          \begin{subfigure}[b]{.3\linewidth}
            \centering\scalebox{.7}{\input{Figures/d2.tikz}}%
            \caption{}\label{fig:d2}
          \end{subfigure}%
          \begin{subfigure}[b]{.3\linewidth}
            \centering\scalebox{.7}{\input{Figures/d3.tikz}}%
            \caption{}\label{fig:d3}
          \end{subfigure}%

          \begin{subfigure}[b]{.25\linewidth}
            \centering\scalebox{.7}{\input{Figures/d4.tikz}}%
            \caption{}\label{fig:d4}
          \end{subfigure}%
          \begin{subfigure}[b]{.25\linewidth}
            \centering\scalebox{.7}{\input{Figures/d5.tikz}}%
            \caption{}\label{fig:d5}
          \end{subfigure}%
          \begin{subfigure}[b]{.25\linewidth}
            \centering\scalebox{.7}{\input{Figures/d6.tikz}}%
            \caption{}\label{fig:d6}
          \end{subfigure}%
		\begin{subfigure}[b]{.2\linewidth}
			\centering\scalebox{.7}{\input{Figures/z5.tikz}}%
			\caption{\label{fig:z-5}}
		\end{subfigure}
          \caption{Figures~(a)--(f) show the six decompositions \(D_1, D_2, D_3, D_4, D_5, D_6\) of \(H^*\);
          		Figure~(g) shows the case where \(s_{3,4}^1=1\) and \(s_{i,j}^k=0\) for \((k,i,j)\neq (1,3,4)\).}
        \end{figure}

	Finally, it is not hard to check that \(\sum_{i=1}^6 r_i = 16 + 2\Sigma\).
	Note that, \(G^*\) has at least \(4+\Sigma\) isolated vertices, hence \(n\geq 4+\Sigma\).
	Since \(n\geq 9\), we have \(n-1 \geq \max\{3+\Sigma,8\}\).
	First, suppose that \(H^*=G\).
	If \(r_i\leq \lfloor (n-1)/2\rfloor\) for any \(i\in\{1,2,3,4,5,6\}\),
	then the statement holds.
	Thus, suppose, for a contradiction, that \(r_i\geq \lfloor (n-1)/2\rfloor + 1 \geq \lfloor \max\{3+\Sigma,8\}/2\rfloor +1\)
	for every \(i=1,2,3,4,5,6\).
	Therefore,
	\begin{align*}
	16 + 2\Sigma =
	\sum_{i=1}^6 r_i &\geq
	6\big(\lfloor \max\{3+\Sigma,8\}/2\rfloor + 1\big) \\
	&\geq  6\big(\max\{(2+\Sigma)/2,4\}+1\big) \\
	&= \max\{6+3\Sigma,24\} + 6 \\
	&= \max\{12+3\Sigma,30\}.
	\end{align*}
	Therefore, we have \(\Sigma \leq 4\) and \(\Sigma \geq 7\), a contradiction.
	Therefore, if \(H^*=G\), \(r_i \leq \lfloor (n-1)/2\rfloor\) for some \(i\in\{1,2,3,4,5,6\}\),
	which implies that \(G\) is a Haj\'os graph.

	Thus, we may assume that \(H^*\neq G\).
	If \(r_i\leq \lfloor (4+\Sigma)/2\rfloor\) for any \(i\in\{1,2,3,4,5,6\}\),
	then \(H^*\) is a proper cycle reducing subgraph of \(G\).
	Thus, suppose, for a contradiction, that \(r_i\geq \lfloor (4+\Sigma)/2\rfloor + 1\)
	for every \(i=1,2,3,4,5,6\).
	Therefore,
	\(
	16 + 2\Sigma =
	\sum_{i=1}^6 r_i \geq
	6\left(\lfloor (4+\Sigma)/2\rfloor + 1\right)
	= 18 + 6\lfloor \Sigma/2\rfloor
	\geq 18 + 6\left(\frac{\Sigma-1}{2}\right)
	\)
	which implies that \(\Sigma\leq 1\).
	Now, if \(\Sigma=0\)
	or one in \(\{s^1_{1,4},s^1_{2,3},s^1_{1,2},s^2_{4,2},s^2_{3,1},s^2_{4,3}\}\)
	is at least \(1\),
	then \(\D_2\) is a cycle decomposition of \(H^*\) of size \(2\).
	If one in \(\{s^1_{1,3},s^1_{2,4},s^2_{4,1},s^2_{3,2}\}\)
	is at least \(1\),
	then \(\D_1\) is a cycle decomposition of \(H^*\) of size \(2\).
	Thus, \(s^1_{3,4}=1\) or \(s^2_{1,2}=1\).
	Suppose, without loss of generality, that \(s^1_{3,4}=1\) and \(S^1_{3,4} = \{v\}\),
	and let \(R_1=Q_1(x_1,v)\) and \(R_2=Q_1(v,x_2)\).
	Let \(t_i\) be the number of internal vertices of \(R_i\) in \(R_i \cap Q_2\),
	for \(i=1,2\).
	Suppose, without loss of generality, that \(t_1\leq t_2\).
	Note that \(G-E(H^*)\) has at least \(5+ t_1+t_2 \geq 5 + 2t_1\) isolated vertices.
	Moreover, \(R_2 + P_{3,4}(v,x_4)+P_{4,1} + P_{1,3} + P_{3,2}\) is a cycle in \(G\),
	and by Lemma~\ref{lemma:double-paths}, \(R_1 + P_{3,4}(v,x_3) + Q_2 + P_{4,2}+P_{2,1}\) (see Figure~\ref{fig:z-5})
	can be decomposed into at most \(1 + t_1\) cycles.
	Thus, \(H^*\) can be decomposed into \(2+t_1\leq \lfloor (5+2t_1)/2\rfloor\) cycles.
	Therefore, \(H^*\) is a \((2+t_1)\)-cycle reducing subgraph of \(G\).
	By the minimality of \(G\), \(G-E(H^*)\) is a Haj\'os graph,
	and Since \(H^*\neq G\), by Lemma~\ref{lemma:cycle-reducing}, \(G\) is a Haj\'os graph.
\end{proof}

\section{Concluding remarks}\label{sec:conclusion}

Conjectures~\ref{conj:gallai} and~\ref{conj:hajos} are important problems that have attracted the attention of many researchers.
Although they have been verified for many classes of graphs, 
there still is space for 
new ideas and techniques.
We believe we can extended the technique presented here to be to deal with these conjectures for other classes of graphs,
for example, graphs with treewidth at most~$4$. 

\bibliography{bibliografia}

\end{document}